\documentclass{article}
\usepackage{pdfsync}  
\usepackage[T1]{fontenc}
\usepackage{amsthm,amsmath,amssymb}
\usepackage{mathrsfs}
\usepackage[colorlinks,citecolor=blue,urlcolor=blue]{hyperref}  

\theoremstyle{definition}

\newtheorem*{thm*}{Theorem}

\newtheorem{thm}{Theorem}
\newtheorem{lem}[thm]{Lemma}
\newtheorem{prop}[thm]{Proposition}
\newtheorem{cor}[thm]{Corollary}
\newtheorem{defn}[thm]{Definition}
\newtheorem{Q}[thm]{Question}

\newtheorem{rmk}[thm]{Remark}

\newtheorem{CardComp}{Cardinal Comparability}
\newtheorem{ISM}{Infinite Sums are Maxima}
\newtheorem{IPM}{Infinite Products are Maxima}
\newtheorem{PermI}{Permutation Invariance}
\newtheorem{CWF}{Cardinalities are Well-Founded}
\newtheorem{mainthm}{(Informal)}
\newtheorem*{FFCT}{Fine's 
Characterization Theorem}

\newtheorem*{FClass}{Fine's Classification Theorem}
\newtheorem{NRCThm}{Natural Relative Categoricity Theorem}

\date{Version submitted July 23, 2016. \\ Please see forthcoming article in \\
	\emph{Notre Dame Journal of Formal Logic} \\ for final version}

\newcommand{\rng}[1]{\mathrm{rng}(#1)}
\newcommand{\dom}[1]{\mathrm{dom}(#1)}
\newcommand{\otherwise}{\textrm{otherwise}}

\newcommand{\nd}{\mathbin{\wedge}}

\newcommand{\xkl}{[X]_{E(\boxminus)}}

\newcommand{\eeq}{=}
\newcommand{\leeq}{\leq}

\makeatletter
\newtheorem*{rep@theorem}{\rep@title}
\newcommand{\newreptheorem}[2]{%
\newenvironment{rep#1}[1]{%
 \def\rep@title{#2 \ref{##1}}%
 \begin{rep@theorem}}%
 {\end{rep@theorem}}}
\makeatother

\newreptheorem{prop}{Proposition}

\begin{document}


\title{Abstraction Principles and the Classification of Second-Order Equivalence Relations}

\author{Sean C. Ebels-Duggan}
\maketitle

\begin{abstract}
This paper improves two existing theorems  of interest to neo-logicist philosophers of mathematics.  The first is a classification theorem due to  Fine 
for equivalence relations between concepts definable in a well-behaved second-order logic.  The improved theorem states that if an equivalence relation $E$ is defined without non-logical vocabulary, then the \emph{bicardinal slice} of any equivalence class---those equinumerous elements of the equivalence class  with equinumerous complements---can have one of only three profiles.  The improvements to Fine's theorem allow for an analysis of the well-behaved models had by an abstraction principle, and this in turn leads to an improvement of Walsh and Ebels-Duggan's relative categoricity theorem. 
\end{abstract}



\section{Introduction}\label{sec:intro}

Neo-logicist philosophers of mathematics are impressed by the fact that  some \emph{abstraction principles} can interpret interesting fragments of mathematics.
These principles are sentences of an enriched second-order language of the following form:  for $E$ an equivalence relation between second-order objects, and $\partial$ a functor taking second-order objects to first order objects, the abstraction principle $A_E[\partial]$ is the sentence
\begin{equation} \label{e:abstprin}
(\forall X, Y)(\partial X = \partial Y \leftrightarrow E(X,Y))\end{equation}
In virtue of their form, say neo-logicists, abstraction principles are eligible to be ``analytic truths''---meaning that they are epistemically near enough to logical truths.\footnote{See most notably~\cite{Wright1983} and~\cite{Hale2001}.}  And if abstraction principles count, in some sense, as near enough to logical, then so should any mathematics they interpret.  

So goes the argument, and not without objections and replies.  Much turns on what would count as a ``(near enough to) logical'' abstraction principle.  But a plausible minimal requirement is this:  an abstraction principle's equivalence relation must itself be ``logical'',\footnote{We put aside concerns, like those of Quine in~\cite{Quine1970aa}, that second-order languages are themselves not ``logical'' in virtue of their quantifying over higher-order objects.}   and so cannot require identification of particular objects, concepts, or relations for its definition.  The thinking is that logic is indifferent to particulars, whether objects, concepts (monadic second-order objects), or relations (polyadic second-order objects). So for an abstraction principle to be logical, it is necessary that its equivalence relation be accordingly indifferent.  A natural way to flesh out this notion of ``indifference'' is with the notion of \emph{permutation invariance}:  the status of the relation doesn't vary, no matter how one exchanges the objects of its concern.\footnote{This view is associated with Tarski's identification of logical notions as those that are permutation invariant in this sense (see~\cite{Tarski1986aa}), for a more recent defense see the work of Gila Sher (see~\cite{Sher1991}).  Other relevant discussions can be found in~\cite{Bonnay2008aa} and~\cite{McGee1996}.  The so-called Tarski-Sher thesis identifying logicality with permutation invariance is controversial; that permutation invariance is a necessary feature of logical notions is not (see~\cite{sep-logical-constants}).}

Further complicating the neo-logicist's hopes is the fact that not all abstraction principles---even those based on permutation invariant equivalence relations---\emph{should} count as logical:  so not only does said neo-logicist need an account of how an abstraction principle could be logical, but that account must also sort the good principles from the bad.  
So more needs to be said for what kinds of logical equivalence relations there are on concepts, and which are apt to   yield 
suitably logical abstraction principles.  
It is thus of interest to classify logical equivalence relations on concepts.

This paper proves a classification theorem for such equivalence relations.  Though discovered independently, it happens that the classification theorem here presented is a stronger version one already on the books.  Kit Fine's Theorem~4 of~\cite[p.~142]{Fine2002} classifies infinite concepts in standard models.  The theorem is then used to determine the finest abstraction principle satisfiable on all infinite standard models~\cite[Theorem~6, p.~144]{Fine2002}.  

But Fine's theorem in its given form hides its true power;  hence our new presentation.
In our version the result is put in a deductive setting in which cardinalities are well-behaved.  This is relatively minor, but it allows us to apply the theorem to problems cast in just such deductive settings (about which more in a moment).  More importantly, our version of the theorem sorts equivalence relations, and their abstraction principles, more usefully.  In other words:  there are abstraction principles well-discussed in the literature on neologicism, but as Fine states the theorem in~\cite{Fine2002}, one needs to squint to see how their equivalence relations are classified.  The  version here given allows for more clear-eyed recognition of this sorting, and allows for a generalization of the classification to (Dedekind) finite concepts as well.  

It is something of a journey from this restatement and expansion of Fine's theorem to its payoff, but the midpoint is a worthy stop.  The improved theorem enables an analysis of when a given abstraction principle has a well-behaved model (not just a standard model).  The existence of such models, as we shall see, are determined by the classification of the principle's equivalence relation on a certain sub-class of concepts---the \emph{bicardinally equivalent} ones.  This at least gives the neo-logicist a tool for analyzing the space of abstraction principles.  

The terminal payoff comes in the form of a more direct result:  the restated and expanded  theorem allows an improvement of the relative categoricity theorem given by Walsh and Ebels-Duggan in~\cite{WalshED2015}.  Walsh and Ebels-Duggan showed that abstraction principles are \emph{naturally relatively categorical} when and only when their equivalence relation is coarser than that of the neo-logicist's favored abstraction principle, {\tt HP}.  Using the improved version of Fine's theorem, we show that this result obtains even when we remove the qualifier ``naturally''.  This improvement allows for a plausible case to be made that {\tt HP} and its ilk pass a minimal threshold for logicality, since (unqualified) relative categoricity is arguably the correct notion of permutation invariance for abstraction principles.

The remainder of this paper is organized as follows.  Sections~\ref{sec:moreintro}--\ref{ssec:bceq} are introductory, motivating the project and setting in place the formal machinery for the theorems.  
Section~\ref{sec:moreintro} explains why Fine's theorem could use improvement, and gives an informal characterization of the improved theorem, which we call the ``\ref{main}''.
Section~\ref{ssec:background} describes the second-order language in which we will work, and its structures; as well as listing our use of standard abbreviations and  giving formal clarifications of the terms we have loosely defined above and in section~\ref{sec:moreintro}.  Here also we describe the cardinality assumptions we adopt with slight modification from~\cite{WalshED2015}, explaining the sense in which our background logic admits only  ``well-behaved cardinalities''.
The next section~\ref{ssec:bceq} 
explains why restricting our attention to bicardinally equivalent concepts is apt to our purposes.  

With preliminaries done, 
section~\ref{sec:mainthm} provides a rigorous statement and proof of the \ref{main}.  
The remaining sections discuss the relevance of the \ref{main} to the neo-logicist project.  Our aim in these sections will not be to argue for or against a particular version of neo-logicism, but instead to let the \ref{main} shed light on some technical questions of interest.
Section~\ref{sec:badco} will address the sorting of so-called ``bad companions'' by the \ref{main}; 
while section~\ref{sec:relcat} proves and discusses the extension of the relative categoricity theorem of~\cite{WalshED2015}.  For ease of exposition, we put off some of the more tedious or repetitive proofs to appendices, we conclude with these.

\section{Why bother a sleeping theorem?}\label{sec:moreintro}

The restatement and expansion of Fine's theorem is more than just a curious exercise, for we believe there is a \emph{better} statement of the theorem than the original.  We begin by giving Fine's version.

To state Fine's theorem efficiently, we need a bit of notation which we will use in the rest of the paper.  Letting ``$\cup$'', ``$\cap$'', and ``$-$'' denote the usual boolean functions on sets of union, intersection, and difference, we will use ``$X\triangle Y$'' to mean the \label{page:notation} \emph{symmetric difference} between $X$ and $Y$, namely the set $(X-Y) \cup (Y-X)$.  And we'll denote the complement of $X$ in a universe $M$ by $M-X$.  Finally, we'll say that $X$ and $Y$ are \emph{bicardinally equivalent} to mean that $|X|=|Y|$ and $|M-X|=|M-Y|$;  that is, that $X$ and $Y$ are equinumerous (in a given structure $\mathcal M$) and that their complements are also equinumerous (also in the given structure).  We'll abbreviate bicardinal equivalence by writing ``$X \boxminus Y$''.  Note that $\boxminus$ is (provably) an equivalence relation in second-order logic.

Now, to  Fine's theorem:  Let $\mathcal M$ be 
a standard model of second-order logic, $X$ and $Y$ be infinite concepts such that $X\boxminus Y$, and $E$ be a permutation invariant equivalence relation. The concepts $X$ and $Y$ are \emph{representative} just if 
\begin{equation}\label{e:Fine-rep} \mathcal M\models (\forall Z,W)(Z, W \boxminus X \nd E(X,Y)\rightarrow E(Z,W)) \end{equation}

\begin{FClass}\label{fclass}
Given a standard model $\mathcal M$, infinite concepts $X$ and $Y$ of $\mathcal M$ are representative if and only if they meet exactly one of the following conditions.
\begin{enumerate}
\item $|X\triangle Y|= |X|<|M|$.   \label{i:fclassi} In Fine's words, $X$ and $Y$ are ``small but very different''.
\item   $\omega \leq |M-X|,|M-Y|<|X|=|M|$ and $|(M-X)\triangle (M-Y)|=|(M-X)\cup (M-Y)|$.  
\label{i:fclassii}  
In Fine's words this is that $X$ and $Y$ ``are almost universal but with infinite very different complements.''
\item $|X|=|M-X|=|M|=|M-Y|=|Y|$, and $|X\triangle Y| = |X\triangle (M-Y)| = |X|$ or vice versa.  In  Fine's words, \label{i:fclassiii} $X$ and $Y$ ``are bifurcatory with one very different from the other and from its complement.''
\end{enumerate}
\end{FClass}
This statement of Fine's theorem is somewhat dizzying;  Fine's statement is more succinct since he uses definitions we have spelled out.  But the complexity of the theorem is not by itself a count against it.  The reason it can be improved, however, is that its classes do not obviously organize the variations of equivalence relations at play in the investigation of abstraction principles.  
The above statement of Fine's theorem also obscures a generalization of the theorem:  in fact there is a general version of Fine's theorem that applies to finite, as well as infinite, concepts.  It is hard to see how this could be given the classification Fine offers.

Restating Fine's theorem will make it easier to use, and highlights the connections between the resultant classification and abstraction principles of particular interest to neo-logicists.  
(Unfortunately, restating and expanding the theorem to the finite case requires proving it anew.  One would hope to rely on Fine's proof, but in fact once the work is done for the extension, the initial theorem is all but proved.)

To improve on Fine's classification theorem, we start by thinking about three types of abstraction principles and their similarities.

We noted in the Introduction~\ref{sec:intro} that Frege's logicist project was to show that arithmetic is in some sense ``really'' logic by showing that arithmetic laws are in fact logical laws.  Frege executes his program in two steps:  first by proving in second-order logic (without our cardinal assumptions) that {\tt BLV},
\begin{equation}\label{e:BLVdef} (\forall X,Y)(\varepsilon X = \varepsilon Y \leftrightarrow X = Y ) \end{equation}
implies {\tt HP}:
\begin{equation}\label{e:HP} (\forall X,Y)(\# X = \# Y \leftrightarrow X \approx Y) \end{equation}
Here ``$=$'' indicates co-extensiveness of concepts, and ``$\approx$'' indicates equinumerousity, the existence of a bijection between concepts.  Note that in the notation of (\ref{e:abstprin}), {\tt BLV} is $A_=$ and {\tt HP} is $A_\approx$.

Frege then shows 
that in second-order logic (with full comprehension), {\tt HP} interprets (what is now called) second-order Peano Arithmetic.  In fact, relatives of {\tt HP} will do the same thing,\footnote{Interpreting Peano Arithmetic does not require the full power of {\tt HP}, but only that an abstraction principle match {\tt HP} on all finite concepts.  See~\cite{Heck1997ab}, and~\cite{MacBride2000aa} for further discussion.} including the \emph{bicardinality principle}, {\tt BP}, which is $A_\boxminus$.  So if {\tt BLV} is logical, then so  are {\tt HP} and {\tt BP}, and thus so is arithmetic.  

But {\tt BLV} is no principle of logic, since it is quite famously inconsistent in second-order logic, deriving the Russell paradox.  Neo-logicism proceeds against Frege's objections in~\cite[\S 63ff]{Frege1980}, asserting arithmetic is logical from the second part of the program alone---that is, that as {\tt HP} is plausibly logical, so then is arithmetic.\footnote{Wright's~\cite{Wright1983} marks the staring point of neo-logicism, though the observation that Frege's program proceeds in two parts was made in~\cite[p.~183, p.~194]{Parsons1965}.}

At a minimum, abstraction principles must be consistent to count as logical.  But that this minimum 
is not enough is the ``bad company'' problem:  there are consistent abstraction principles that are otherwise unacceptable, at least to the neo-logicist.  
The prototypical bad companion is {\tt NP}:\footnote{ 
The first bad companions appeared in~\cite{Boolos1990};  one of these was simplified into {\tt NP} by Wright in~\cite{Wright1997aa}.}  
\begin{equation}\label{e:NP} (\forall X, Y)(\eta X = \eta Y \leftrightarrow |X\triangle Y| < \omega) \end{equation}
The principle {\tt NP} is unwelcome to the neo-logicist because in contexts of well-behaved cardinalities, it implies that the universe is Dedekind finite.\footnote{And this is unwelcome to neo-logicists because they claim an abstraction principle shouldn't imply anything about sortal concepts unrelated to the abstracts;  see~\cite[p.~415]{Hale2000aa}, \cite[pp.~295--7]{Wright1997aa} and~\cite[pp.~314--5]{Wright1999}, as well as \cite{ED-2015}.
    }  

Another bad companion, discussed in~\cite[\S~5.5]{WalshED2015}, is the Complementation Principle, {\tt CP}, 
\begin{multline*}
 (\forall X, Y)( \copyright X =  \copyright Y \leftrightarrow {} \\
 [ (|X|=|Y|=|M-X| \wedge (X=Y \vee X = M-Y) ) \vee {} \\ (|X|\neq |M-X| \wedge |Y|\neq |M=Y|)] ) 
\end{multline*}
This principle sorts only concepts the same size as their complements.  Such concepts are grouped together with only their complements, and all other concepts are grouped in a ``junk'' equivalence class.  As we will show below, this principle counts as ``bad company'' as well (see section~\ref{ssec:LCP}).

Taking stock, we have so far two good abstraction principles in {\tt HP} and {\tt BP}, and three bad ones ({\tt BLV}, {\tt NP}, and {\tt CP}).  Can we sort these into useful classes?  
We can if, like Fine, we look at their behavior just on bicardinally equivalent concepts.    Clearly, when restricted to just bicardinally equivalent concepts, the equivalence relations of {\tt HP} (equinumerousity) and {\tt BP} are \emph{trivial}:  they include all such concepts.

Second, {\tt BLV} and {\tt NP} are, in a sense, of the same type:  the equivalence relations for both are concerned with the size of the symmetric difference between the two related concepts.  The principle {\tt BLV} discriminates concepts if their symmetric difference has non-zero size.  Likewise {\tt NP} discriminates if the symmetric difference is Dedekind infinite.  Say such equivalence relations that sort concepts according to the size of their symmetric difference, are \emph{separations}:  two concepts with ``few'' objects falling under  one but not both of the two concepts are equivalent;  two concepts with ``many'' such objects are not equivalent.  

Lastly, we have {\tt CP}.  Such an equivalence relation is neither trivial, nor a separation, but it has many of the drawbacks of separations.  And such an equivalence relation can be generalized, as {\tt NP} is a generalization of {\tt BLV}:   two bicardinally equivalent concepts can be grouped together if their symmetric difference is ``small'', or the complement of their symmetric difference is ``small''.  
So let us say that on bicardinally equivalent concepts, such relations are \emph{complementations}:  they sort concepts by the size of their symmetric difference, or the size of the complements of their symmetric difference.

Finally, say that an equivalence relation $E$ is a \emph{refinement} of an equivalence relation $E'$ if whenever two concepts are $E$-equivalent, they are $E'$ equivalent.\footnote{See~\cite{Fine2002} and~\cite{Antonelli2010aa} for discussion of ``finer'' and ``coarser'' equivalence relations and their relevance to neo-logicism.  Cook's \cite{Cook2016doi} standardizes the terminology and indicates relations between different kinds of invariance. It was Cook's paper that made me see the connection to Fine's work.  
    }

We can now state the informal (though somewhat inexact) version of our main theorem:

\begin{mainthm}\label{main}
Let $E$ be a purely logical equivalence relation, and let our background logic be a strong but natural version of second-order logic.  If we look only at $E$ on bicardinally equivalent concepts, then $E$ is either the trivial equivalence relation, or it is a refinement of a 
separation, or it is a refinement of a 
complementation.  
\end{mainthm}

This version of Fine's result makes more obvious the relationship between its classification scheme and the kinds of abstraction principles that have been seen in the neo-logicist's laboratory.  
In restating the theorem we observe a striking alignment:  the problematic abstraction principles arise from \emph{non-trivial} equivalence relations.  This is no accident;  as we will see in section~\ref{sec:badco}, abstraction principles involving non-trivial equivalence relations limit the size of their models in just the way that is typical of bad company.

\section{The language $L_0$, comprehension, and cardinality assumptions} \label{ssec:background}

Our background logic is a basic, though robust, version of second-order logic, adopted with slight weakening from \S 2 of~\cite{WalshED2015}.  In short, the language $L_0$ uses lower-case letters to vary over first-order objects, and upper-case letters to vary over second-order objects of all finite arities (arity will be clear from context in our presentation)---such objects of singular arity are called ``concepts'' or ``sets'', others are called ``relations''.
All terms of $L_0$ are variables;  we exclude constants of either type.  Formulae of $L_0$ and their interpretation are as usual.  

Thus models of this language are of the form 
\begin{equation}
\mathcal M = ( M, S_1[M], S_2[M], \ldots )
\end{equation}
where $M$ is non-empty and the members of $S_i[M]$ are subsets of $M^i$ for $i\geq 1$.   We do not require $S_i[M]$ to be the full power-set of $M$, 
for compatibility of our results with those of~\cite{WalshED2015};  that is, we work in the non-standard semantics. 
Likewise we require our models to satisfy 
comprehension axioms for all formulae
 in their signature; 
these 
axioms are of the form
\begin{equation}
(\exists X)(\forall y)(Xy\leftrightarrow \Phi(y, \bar p, \bar R))
\end{equation}
where $\bar p$ is a sequence of parameters from $M$, and $\bar R$ is a sequence of relation parameters from $\bigcup_{i\in \mathbb{N}}S_i[M]$. This ensures that all finite concepts, and many more besides, are included in $S_1[M]$.  The remaining axioms of our logic, which we will call ``second-order logic'' are those of {\tt D2} found in~\cite[pp.~65--7]{Shapiro1991}, but excluding the axiom of choice.  
We will soon be augmenting this logic with choice principles to ensure well-behaved cardinalities. .

Most symbols used are standard;  we rehearse a few:  The symbol $\sqcup$, which appears in Proposition~\ref{p:nonA}, means the disjoint union;  we will at times abuse notation using it to mean the union of two disjoint sets.
As usual we use ``$\preceq$'' to assert the existence of an injection, and ``$\approx$'' the existence of a bijection.  We will also use expressions like ``$|X|\leq |Y|$'' and ``$|X|=|Y|$'', and we will use the convention of writing, e.g., $X\stackrel{f}\preceq Y$ to mean that $f$ is an injection from $X$ into $Y$; likewise with the other expressions.  As above  the expression ``$\boxminus$''  indicates the relation of \emph{bicardinality};
the expression ``$X\stackrel{f}{\boxminus} Y$'' means that $f$ is a bijection from $M$ to itself and $f(X)=Y$ and $f(M- X) = M- Y$.  Both $\approx$ and $\boxminus$ are provably equivalence relations in second-order logic;  we will say in the obvious circumstances that concepts are cardinally, or bicardinally, equivalent.

We identify concepts as infinite \emph{in $\mathcal M$} if they are Dedekind infinite, that is, that there is in $\mathcal M$ an injection from the concept into a proper subconcept of itself, and we write this with the expression $|X|\geq \omega$.  We say that $X$ is finite in $\mathcal M$ if it is not Dedekind infinite, writing it $|X|<\omega$. 
These expressions are used exclusively within the model $\mathcal M$.  When we wish to say that a set $X$ is finite or infinite in the metatheory, we will either say so explicitly, or use ``$|X|=n$ for some $n\in \mathbb{N}$'' or ``$|X|<|\mathbb{N}|$''.  

We will abuse notation and use ``$M$'' for both the first order domain of $\mathcal M$ and the universal concept $\{x \mid x = x\}$; 
and we'll use functional notation for relations that are functional in $\mathcal M$. 
In general we will write $f\in \mathcal M$ to mean that for some $i\in \mathbb{N}$, 
$f$ is an $i$-ary function and $f\in S_{i+1}[M]$; similarly with concepts and relations.  
We write $f(X)$ to mean the image of $X$ under $f$;  such a relation is always in $\mathcal M$ by comprehension.

Though $L_0$ is expressively rich, it is not too rich, for (crucially for our results), any function $\pi:M\rightarrow M$ that extends to  permute objects \emph{of all types in $\mathcal M$} is an automorphism of $\mathcal M$.  This can be seen by the fact that, for a given $\mathcal M$ and $\pi:M\rightarrow M'$, the ``push model of $\mathcal M$ under $\pi$'' is always isomorphic to $\mathcal M$, with $\pi$ the witnessing isomorphism.\footnote{See \cite[pp.~225--31]{Button2013aa} for more on this method.} %
If $\pi$ is also such as to make $\mathcal M$ the same structure as its ``push model'', then such a $\pi$ is an automorphism.  By comprehnsion, if $\pi:M\rightarrow M$ is a bijection in $\mathcal M$, then $\pi$ meets the needed criteria.  Thus, we may state the relevant result in the following form for our use:

\begin{PermI}\label{p:perminv}  If $\pi:M\rightarrow M\in \mathcal M$ is bijection and 
$\Phi(\overline x, \overline X)$ is an $L_0$-formula, then 
\begin{equation} \label{e:perminv} \mathcal{M} \models (\forall \overline x, \overline X)(\Phi(\overline x, \overline X) \leftrightarrow \Phi(\overline{\pi(x)}, \overline{\pi(X)})) \end{equation}

Under the same hypotheses about $\pi$, if 
$E(X,Y)$ is an $L_0$-definable  equivalence relation on concepts of $M$, then 
\begin{equation}\label{e:perminvpi} \mathcal{M} \models E(X,Y) \leftrightarrow E( \pi (X), {\pi} (Y)) \end{equation}
\end{PermI}
Here, of course, $E$ being \emph{$L_0$-definable} means definable \emph{without} parameters:  the formula defining $E$ has only the second-order variables $X$ and $Y$ free.  Definability via an $L_0$-formula is the technical correlate of the informal ``purely logical'' used in Section~\ref{sec:intro}.   
Our main results depend on \ref{p:perminv}, in the sense that they obtain for any equivalence relation that is permutation invariant, 
not just those $L_0$-definable.

The second way our background logic is robust is that we require that the relation of equinumerousity behave as it does in more familiar contexts, e.g., {\tt ZFC}.  In particular, we require the following  
to obtain in all models $\mathcal M$  under consideration:

\begin{CardComp}\label{cc}
\begin{equation}\label{e:cardcomp}
\mathcal M\models (\forall X,
Y)(X\leq Y \vee Y\leq X)
\end{equation}
\end{CardComp}

\begin{ISM}\label{ism}
\begin{equation}\label{e:ism}
\mathcal M\models (\forall X, Y)(|X|,|Y|\geq \omega \rightarrow |X \sqcup Y| = \max(|X|, |Y|)
\end{equation}
\end{ISM} 

\begin{IPM}\label{ipm}
\begin{equation}\label{e:ipm}
\mathcal M\models (\forall X, Y)(|X|,|Y|\geq \omega \rightarrow |X \times Y| = \max(|X|, |Y|)
\end{equation}
\end{IPM}

\begin{CWF}\label{cwf}
For any 
$L_0$-formula $\Phi(X)$, 
\begin{equation}
\mathcal M \models (\exists X)(\Phi(X)) \rightarrow (\exists X)(\forall Y)(\Phi(Y) \rightarrow X \preceq Y)
\end{equation}
\end{CWF}
We take over these principles  
nearly directly from~\cite{WalshED2015};  
both for their handiness and as we will apply the~\ref{main} to offer a solution to some open questions from that paper (see section~\ref{sec:relcat}).\footnote{In~\cite{WalshED2015} the principles \ref{cc}, \ref{ism}, and \ref{ipm}  were deployed as consequences of the principle {\tt GC}, though {\tt GC}'s well-ordering was in the main results only to show that the restrictions to small concepts and to abstracts were required (see~\cite[equations 4.22--4.23 and following]{WalshED2015}).  In this paper we will use only the cardinal consequences of {\tt GC} listed above.  Our logic is also weaker than that deployed in \cite{WalshED2015} in that we make no use of the principle {\tt AC}.
}  
Thus, in what follows and unless otherwise noted, all structures satisfy full comprehension and these cardinality assumptions.

The principles \ref{cc} and \ref{ism} are ubiquitous in the proof of our first version of the \ref{main}, Theorem~\ref{thm:main-0formal}.  The schematic principle \ref{cwf} is used in moving from Theorem~\ref{thm:main-0formal} to the second  formal version of the \ref{main}, Theorem~\ref{thm:main-formal}.  The principle \ref{ipm} is deployed mainly in the form of a pairing function in section~\ref{sec:relcat}.

We will need specific notation suited to dealing with the divergence, in non-standard semantics, between Dedekind finitude and finitude in the metatheory.  As this notation pertains only to the proof of the formal version of the \ref{main}, we introduce it in  section~\ref{sec:mainthm}.

\section{Bicardinal equivalence and permutation invariance}\label{ssec:bceq}

Our \ref{main} says that an $L_0$-definable equivalence relation, when restricted to bicardinally equivalent concepts, can have one of only three profiles.  But why is it interesting to look at \emph{this} particular restriction of such equivalence relations?  This portion of our paper will motivate this restriction.  

Our interest in $L_0$-definable equivalence relations is that, provided that second-order languages are 
a part of logic, these relations meet at least one criterion of logicality:  they are permutation invariant.  
But one interesting fact about second-order logic is that  there are $L_0$-definable properties that can \emph{distinguish} second-order objects.  To see what this means, consider the formula 
\begin{align}
Singleton(X):{} & (\exists x)(\forall y)(Xy \leftrightarrow y=x) \label{e:card1} 
\end{align}
The formula $Singleton(X)$ can distinguish between second-order objects in sufficiently rich structures for second-order languages: if a structure includes distinct first-order objects $a$ and $b$, then in that structure $Singleton(\{a\})$ while $\neg Singleton(\{a,b\})$. These second-order objects are thus distinguished by the formula ``$Singleton(X)$''.
Thus, while first-order objects are indistinguishable using only $L_0$-definable notions, second-order objects are not.  
It makes sense, then, that we would want to pay special attention to collections of second-order objects that are \emph{not} distinguishable in this way.  And we need not look far for concepts that are so indistinguishable.  As can be see from \ref{p:perminv},  concepts are bicardinally equivalent if and only if they cannot be distinguished using an $L_0$-formula.  
It is thus of interest to see how purely defined equivalence relations behave on these concepts in particular.

Relatedly, bicardinal equivalence is
\emph{indicative} of permutations.
Let $E$ be an equivalence relation on concepts.  We will say that $E$ is an \emph{indicator of permutations} just if, for any second-order structure $\mathcal M$,\footnote{Not just those with well-behaved cardinalities.} 
and any function $f:M\rightarrow M$ such that if $X\in \mathcal M$ then $f(X)\in \mathcal M$, if 
\begin{equation}\label{e:permind}(\forall X,Y)( E(f(X), f(Y)) \rightarrow E(X,Y) ) \end{equation}
in that structure, then $f$ is a permutation (it is a bijection from the first-order domain $M$ to itself).  
\begin{thm}\label{p:permbicard}
Bicardinality is an indicator of permutations.\end{thm}

\begin{proof}
Observe that if $f$ is not injective, then (\ref{e:permind}) fails:   let $x, y$ be such that $f(x)=f(y)$, and note that $\neg(\{x,y\} \boxminus \{f(x)\})$  $f(\{x,y\}) = \{f(x)\} \boxminus \{ff(x)\}$.  So it remains only to show that if $f$ is injective and (\ref{e:permind}) holds in $\mathcal M$, then $f$ is surjective.

Further, if $|M|$ is finite in $\mathcal M$, then every injection on $M$ is a permutation.  So we may further assume that $M$ is not finite in $\mathcal M$.

Working in $\mathcal M$, we show first that
\begin{equation}
\label{e:permbic-cond}
\textrm{if}\; |X|=|M| \; \textrm{then} \; X-f(X) \boxminus f(X-f(X))
\end{equation}
That the concepts are equinumerous follows from the fact the $f$ is injective.  To see that their complements are equinumerous, we have that:
\begin{align}
|M| &\geq |M-(X-f(X))| 						\\
	& = |(M-X)\cup f(X)| \geq |f(X)| = |M|	
\end{align}
where the last equality follows again by the injectivity of $f$.  By the Schr\"oder-Bernstein theorem, $|M-(X-f(X))|=|M|$.\footnote{Note this is provable in unaugmented second-order logic, see~\cite[Theorem 5.2, p.~102ff]{Shapiro1991}).}  An identical argument shows that $|M-f(X-f(X))|=|M|$.

Thus by (\ref{e:permbic-cond}) we have 
\begin{equation}
M-f(M) \boxminus f(M-f(M))
\end{equation}
Since $\boxminus$ is preserved under complementation, 
\begin{equation}
\label{e:permbic-compl}  f(M) \boxminus M-f(M-f(M))
\end{equation} 
which implies that 
\begin{equation}
\label{e:permbic-card}
|M|=|f(M)|=|M-f(M-f(M))|
\end{equation}
So by (\ref{e:permbic-cond}) together with  (\ref{e:permbic-compl}) and (\ref{e:permbic-card}) we obtain that 
\begin{equation}\label{e:picbic-l}
f(M) \boxminus M-f(M-f(M)) \boxminus f(M-f(M-f(M)))
\end{equation}
By assumption we have (\ref{e:permind}), and so $M\boxminus M-f(M-f(M))$ from (\ref{e:picbic-l}). Complementation then gives that $\emptyset \boxminus f(M-f(M))$, which means that $M-f(M) = \emptyset$, and so $f(M) = M$.  Thus $f$ is surjective, and so a permutation. 
\end{proof}

Bicardinally equivalent concepts are thus interesting in their own right.  So it makes sense to observe the behavior of purely defined equivalence relations on just concepts that are $\boxminus$-equivalent.

For this reason we introduce the following notation.  As is customary, given an equivalence relation $E$ and a concept $X$ we denote the $E$-equivalence class containing $X$ by $[X]_E$. 
Looking at how $E$ behaves on the sets $\boxminus$-equivalent to $X$, we say that the \emph{bicardinal slice 
of $E$} is the set of equivalence classes of $\boxminus$-equivalent sets.  More formally, 
\begin{defn}
Let $\mathcal M$ be a model and $E$ an equivalence relation on sets of $\mathcal M$.  Given a set $X$ in $\mathcal M$, we let 
\begin{equation}
E(\boxminus)(X, Y) \Leftrightarrow X \boxminus Y \wedge E(X,Y)
\end{equation}
so that 
\begin{equation}
\xkl  =  [X]_\boxminus \cap [X]_E 
\end{equation}
Thinking of $E$ as a collection of concepts, we can then write
\begin{equation}
E(\boxminus) = \{ [X]_{E(\boxminus)} \mid X\in \mathcal M\}
\end{equation}
We can then think of $E$ as consisting in \emph{bicardinal slices} of the form \begin{equation}
\label{e:E(boxmin)_X}
E(\boxminus)_X = \{ [Y]_{E(\boxminus)} \mid X \boxminus Y\}\end{equation}
 since 
\begin{equation}
E(\boxminus) = \bigcup_{X\in \mathcal M} E(\boxminus)_X 
\end{equation}
\end{defn}
These notions will be used in stating Theorem~\ref{thm:main-formal}, the formal version of the \ref{main}.

\section{The bicardinal classification of equivalence relations}\label{sec:mainthm}

We are now nearly in a position to state our main results formally.  Our theorem will classify equivalence relations at bicardinal slices.  We first must formally state the classes into which any relation can be sorted.  

Theorem~\ref{thm:main-formal} is the formal version of the \ref{main}, which we prove in two stages.  The intermediate stop is Theorem~\ref{thm:main-0formal}, the remaining step is deploying the cardinality assumption \ref{cwf} to obtain the final result.  

Proving these theorems is onerous because non-standard models allow that concepts finite according to $\mathcal M$ might not be finite in the metatheory.  
This distinction is important because the proofs of our main results depend on Lemmata~\ref{lem:class} and~\ref{lem:extend}, which deploy, in the metatheory, finitely many permutations that ``move'' one concept onto another while preserving equivalence classes.  
Thus, we will at times wish to indicate 
for given sets $X$ and $Y$ that  finitely many copies of $Y$ are enough to cover $X$.  
\label{page:dominates} Thus we write $n \times |Y| \geq |X|$ to mean that for some set $Z$ with $|Z|=n\in \mathbb{N}$, $\mathcal M\models |Z\times Y|\geq |X|$.  Similar expressions will be used in an associated way so that their meaning is obvious.  

However, since finiteness in $\mathcal M$ is not the same as finiteness in the metatheory, we cannot count on expected relations of cardinality obtaining.  In particular, \emph{in $\mathcal M$} we have the expected Archimedean property in that 
\begin{equation}\label{e:archimed}
(\forall X, Y)(|X|,|Y|<\omega\rightarrow (\exists Z)(|Z|<\omega \wedge |Z\times Y|\geq |X|))\end{equation}
is a theorem of second-order logic.  But in a non-standard model it is not in general true that for any $X,Y\in S_1[M]$ with $\mathcal M\models |X|,|Y|<\omega$, there is an $n\in \mathbb{N}$ such that $\mathcal M\models n\times |Y| \geq |X|$.  For this reason we introduce the following definition:

\begin{defn}
Given $X, Y\in S_i[M]$, we write $|X|\trianglelefteq |Y|$ to mean there is an $n\in \mathbb N$ such that 
\begin{equation}\label{e:trileft} \mathcal M\models n\times |Y|\geq |X|\end{equation}

Further, we write $|X|\vartriangleleft|Y|$ to mean $|X|\trianglelefteq|Y|$ and $\neg(|Y|\trianglelefteq |X|)$.  Note that $\trianglelefteq$ and $\vartriangleleft$ are relations \emph{in the metatheory};  they are not, in general, expressible in $L_0$.
\end{defn}

In what follows, all concepts are understood to be concepts in  given structures $\mathcal M$. 
In accordance with the two steps towards our main result, we give two definitions for the purposes of classification.  

\begin{defn}\label{def:trivsepcompl0}
Let $E$ be an $L_0$-definable equivalence relation over concepts of a given structure $\mathcal M$.  For any concept $X$, we say that $E(\boxminus)_X$  
\begin{enumerate}
\item is \emph{trivial} if $[X]_{E(\boxminus)} = [X]_\boxminus$.

\item \emph{is separative}, or \emph{refines a separation}, if for all $Y,Z\in [X]_\boxminus$, 
\begin{equation}\label{e:sepv}
	E(Y, Z) \Rightarrow |Y \triangle Z| \vartriangleleft |X|
\end{equation}

\item \emph{is complementative}, or \emph{refines a complementation}, if for all $Y,Z\in [X]_\boxminus$, 
\begin{equation}\label{e:compv}
	E(Y, Z) \Rightarrow (|Y \triangle Z| \vartriangleleft |X| \; \textrm{or}\; |M-(Y\triangle Z)|\vartriangleleft |X| )
\end{equation} 
\end{enumerate}
Obviously if $E(\boxminus)_X$ is separative then it is complementative;  therefore we say that $E(\boxminus)_X$ is \emph{properly separative} or \emph{refines a separation} 
if it is separative and non-trivial, and \emph{properly complementative} or \emph{refines a complementation} 
if it is complementative, non-separative, and non-trivial.
\end{defn}

\begin{rmk}
    If $X$ is properly complementative and $Y\in [X]_{E(\boxminus)}$, then $M-Y\in [X]_{E(\boxminus)}$.
\end{rmk}





We now state the intermediate and final results as:

\begin{thm}
\label{thm:main-0formal}
Let $E$ be an $L_0$-definable equivalence relation over a given structure $\mathcal M$.  Then for any concept $X$, $E(\boxminus)_X$ is either trivial, separative, or  complementative.  

If $|M|>2$ then exactly one of these options holds. 

If $X$ is finite in $\mathcal M$ and $E(\boxminus)_X$ is nontrivial, then either \begin{equation}\label{e:septight} \mathcal M\models E(Y,Z) \Leftrightarrow |Y\triangle Z| \vartriangleleft |X| \end{equation}
or 
\begin{equation} \label{e:comptight}
\mathcal M\models E(Y,Z) \Leftrightarrow \left[ |Y\triangle Z| \vartriangleleft |X| \; \textrm{or}\;  |M-(Y\triangle Z)|\vartriangleleft |X|\right]  \end{equation}
\end{thm}

\begin{thm}[Formal Main Theorem]
\label{thm:main-formal}  
Let $E$ be an $L_0$-definable equivalence relation over a given structure $\mathcal M$.  Then for any concept $X$, $E(\boxminus)_X$ is either trivial, separative, or complementative.

If $|M|>2$ then exactly one of these options holds.  

Further, if $\mathcal M \models |X|<\omega$ then $[X]_{E(\boxminus)}$ is either $[X]_\boxminus$, or $\{ X\}$, or $\{X, M-X\}$.
\end{thm}
As we have said above, this is a strong version of Fine's theorem---as one can see with the aid of a few pages of Venn Diagrams.

\subsection{Non-Archimedian arithmetic}\label{ssec:nonA}

The proof of Theorems~\ref{thm:main-0formal} and~\ref{thm:main-formal} will require, for a given $X$, the use of finitely (in the metatheory) many permutations, each of which fixes some member of $[X]_{E(\boxminus)}$.  These permutations will ``shuttle'' portions of a given set $Z$ onto its image under a bijection $f:Z\rightarrow X$.

We will deploy the relations $\trianglelefteq$ and $\vartriangleleft$ in the next section to prove Lemma~\ref{lem:class}, which is crucial to proving Theorem~\ref{thm:main-formal} in its full generality.  For this we will need to show that certain expected ``arithmetic'' relations hold.  Though tedious to state, we provide them as follows.

\begin{prop}  \label{p:nonA} The following hold for all $X,Y,Z,W\in \mathcal M$:
\begin{align}
\label{e:nonAcomp}
|X|\trianglelefteq |Y|\; &\textrm{or}\; |Y|\trianglelefteq|X|\\
\label{e:nonAaddit}
|X| \trianglelefteq |Y| \;\textrm{and} \; |Z|\leq|W| &\Rightarrow |X\sqcup Z| \trianglelefteq |Y\sqcup W| \\
\label{e:nonAsubt}
 (|Z|=|W| \;\textrm{and} \; |X\sqcup Z| \vartriangleleft |Y\sqcup W| &\Rightarrow |X|\vartriangleleft |Y|) \\
\label{e:nonAtrans}
|X|\trianglelefteq |Y| \trianglelefteq |Z| &\Rightarrow |X|\trianglelefteq |Z| \\
\label{e:nonAsplit}
|X|\trianglelefteq |Y\sqcup Z| &\Rightarrow ( |X|\trianglelefteq |Y| \;\textrm{or}\; |X|\trianglelefteq |Z|) \\
\label{e:nonAsplcor}
|X| = |Y\sqcup Z| & \Rightarrow |X|\trianglelefteq|Y|\; \textrm{or} \;|X|\trianglelefteq|Z|\\
\label{e:nonApsplit}
|X|\vartriangleleft |Y\sqcup Z| &\Rightarrow (|X|\vartriangleleft |Y| \;\textrm{or}\; |X|\vartriangleleft |Z|) \\
\label{e:nonApslcor} 
|X|\vartriangleleft |X\sqcup Y| & \Rightarrow |X|\vartriangleleft|Y|\\
\label{e:nonAlsplit}
|Y\sqcup Z|\trianglelefteq|X| &\Rightarrow |Y|\trianglelefteq|X| \; \textrm{and}\; |Z|\trianglelefteq|X| \\
\label{e:nonAplsplit}
|Y\sqcup Z|\vartriangleleft|X| &\Rightarrow |Y|\vartriangleleft|X| \; \textrm{and}\; |Z|\vartriangleleft|X| \\
\label{e:nonAexp}
|X|\vartriangleleft|Y| & \Rightarrow |X|\vartriangleleft |Y\sqcup Z|\\
\label{e:nonAsubl}
|X\sqcup Y| \trianglelefteq |Z| \;\textrm{and} \; |W|=|Y| &\Rightarrow |X\sqcup W| \trianglelefteq |Z| \\
\label{e:nonAsubr}
|X| \trianglelefteq |Y \sqcup Z| \;\textrm{and} \; |W|=|Y| &\Rightarrow |X| \trianglelefteq |W\sqcup Z|\\
\label{e:nonAsublp}
|X\sqcup Y| \vartriangleleft |Z| \;\textrm{and} \; |W|=|Y| &\Rightarrow |X\sqcup W| \vartriangleleft |Z| \\
\label{e:nonAsubrp}
|X| \vartriangleleft |Y \sqcup Z| \;\textrm{and} \; |W|=|Y| &\Rightarrow |X| \vartriangleleft |W\sqcup Z|
\end{align}
\end{prop}
Their proofs are even more tedious, so we relegate them to Appendix~\ref{appendix:props}.

The following two propositions relate infinity, in both its meta- and intra-theoretic senses, to the relation $\vartriangleleft$, and are easy to prove using (\ref{e:nonAcomp}) and \ref{ipm}.
\begin{prop}
\label{p:easytoprove}
For all $X,Y\in \mathcal M$, 
\begin{equation}\label{e:easytoprove}|X|\vartriangleleft|Y| \Leftrightarrow {} \; \textrm{for all}\; n\in \omega, \mathcal M\models n\times |X|<|Y|\end{equation}
\end{prop}
\begin{prop}\label{p:triinf}
If $\mathcal M\models |X|\geq \omega$, then 
\[ |Y|\vartriangleleft|X| \Leftrightarrow \mathcal M\models |Y|<|X|\]
\end{prop}

\subsection{Preparatory results}\label{ssec:prepR}

Our route to proving Theorem~\ref{thm:main-formal} is via Lemma~\ref{lem:class}, which is proved in the next section.  This lemma specifies sufficient conditions for when $E(\boxminus)_X$ is trivial.  The main way that we prove Lemma~\ref{lem:class} is by exploiting the following consequence of \ref{p:perminv}:

\begin{prop}
\label{p:fixes}
Let $\pi\in \mathcal M$ be a permutation of $M$ such that $\pi$ fixes $X$.  Then for any $Y\in [X]_E$, $\pi(Y)\in [X]_E$.
\end{prop}
For by \ref{p:perminv}, $Y\in [X]_E \Leftrightarrow \pi(Y)\in [\pi(X)]_E = [X]_E$, since $\pi$ fixes $X$.  
Intuitively, the idea is that we can show sets $X$ and $Y$ to be $E$-equivalent if $X$ is equivalent to some $Z$, and there is a permutation fixing $X$ but sending $Z$ to $Y$.

We introduce the following notation to shorten our exposition.
\begin{defn}
Let $f:X\rightarrow Y$ be a bijection with $X$ and $Y$ disjoint.  The we say that the \emph{permutation induced by $f$}, $\iota(f):M\rightarrow M$, is the function defined by 
\[ \iota(f)(x) = \begin{cases}
f(x) & x\in X \\
f^{-1}(x) & x\in Y \\
x & \textrm{otherwise}
\end{cases}
\]
\end{defn}
Notice that if $f\in \mathcal M$ then $\iota(f)\in \mathcal M$ is a well-defined bijection.

Our  goal in this subsection is to prove Lemmata~\ref{p:AtoX}, \ref{p:ExttoS},  \ref{p:ExtoX}, and \ref{p:AtoS}, the proofs of which exploit Proposition~\ref{p:fixes}.   Lemmata~\ref{p:AtoX} and \ref{p:ExttoS} are those with the most involved proofs. 
In the remainder of this section, 
where context makes subscripts unnecessary we will use ``$[X]$''  to indicate $[X]_E$.

\begin{lem}\label{p:AtoX}
Suppose there is a $Y\neq X$ with $Y\in [X]
$ and let $Z\subset Y- X$ with $\mathcal M\models Z\stackrel{f}\preceq X- Y$.  Suppose further that at least one of the following obtains:
\begin{align}
\label{e:AXS}
|Z|&\trianglelefteq |X\cap Y|\\
\label{e:AXExt}
|Z|&\trianglelefteq |M- (X\cup Y)|
\end{align}

Then there is a permutation $\pi\in \mathcal M$ such that $[\pi X]=[X]$ and $Z\subset \pi(X- Y)\subset \pi(X)$, and \begin{equation}
\label{e:Exfixes}  \pi(x)\neq x \Leftrightarrow x\in Z\cup f(Z)
\end{equation}
\end{lem}

\begin{proof}
The proof is essentially the same whether (\ref{e:AXS}) or (\ref{e:AXExt}) obtains.  We set $W$ to be $X\cap Y$ if the former obtains, and $M- (X\cup Y)$ if not.  We reason in $\mathcal M$.

As in the proposition, let $Z\stackrel{f}{\preceq} X- Y$, and let $Z \stackrel{g}{\preceq} n \times W$.  The function  $g$ induces for each $j<n$ a partial function $h_j:W\rightarrow Z$ defined by $h_j(w) = g^{-1}(j,w)$.  Note that each $h_j$ is injective from its domain to its range, which are disjoint;  so $\iota(h_j)$ is well-defined for each $j<n$.

We define for each $j$ three permutations, as follows:
\begin{align}
\label{e:pj1}
p_j & = \iota (h_j) \\
\label{e:qj1}
q_j & = \iota(f\upharpoonright \rng{h_j}) \\
\label{e:rj1}
r_j(x) & =\begin{cases}
q_j(p_j(x)) & x\in \dom{h_j} \\
p_j(q_j(x)) & x \in f(\rng{h_j}) \\
x &\otherwise
\end{cases} 
\end{align}
So  $p_j$ switches $\rng{h_j}$ and $\dom{h_j})$, $q_j$ switches $\rng{h_j}$ and $f(\rng{h_j})$, and $r_j$ switches $f(\rng{h_j})$ and $\dom{h_j}$.

Note that $p_j$ fixes $Y$, $q_j$ fixes $p_j(X)$, and $r_j$ fixes $q_j(p_j(Y))$;  thus three applications of Proposition~\ref{p:fixes} yield that for $s_j = p_j\circ q_j\circ r_j$, \begin{equation}
\label{e:jeqclass} s_j(X), s_j(Y)\in [s_j(X)] = [X]
\end{equation}

Note further that for $w\in \dom{h_j}$, 
\begin{equation}\label{e:sfixes}
s_j(w) = r_j(q_j(p_j(w))) = q_j(p_j(q_j(p_j(w)))) = p_j(p_j(w)) = w
\end{equation}
where the third equality is due to the fact that $q_j\upharpoonright \dom{h_j}$ is the identity function.  It follows from this and the construction of $p_j,q_j$, and $r_j$ that if $s_j(x) \neq x$, then $x\in \rng{h_j} \cup f(\rng{h_j})$.  Conversely, if $x\in \rng{h_j}\subset Y- X$, then $s_j(x)\in X- Y$, and if $x\in f(\rng{h_j})\subset X- Y$, then $s_j(x)\in Y- X$.  Thus if $x\in \rng{h_j}\cup f(\rng{h_j})$ then $s_j(x)\neq x$.  So we have established 
\begin{equation}
\label{e:jfix}
s_j(x)\neq x \Leftrightarrow x\in \rng{h_j} \cup f(\rng{h_j})
\end{equation}

Set $\pi = s_0 \circ \ldots \circ s_{n-1}$, and note that $Z = \bigcup_{j<n}\rng{h_j}$. So we have
\begin{align} \label{e:AX1}
\pi(X- Y) &  = s_0 \circ\ldots\circ s_{n-1}\left(X- Y\right) \\
\label{e:AX4}& =  s_0 \circ\ldots\circ s_{n-1}\left(\left(\left(X- Y\right)- f(Z)\right) \cup f(Z)\right) \\
\label{e:AX5}& =  ((X- Y)- f(Z)) \cup s_0 \circ\ldots\circ s_{n-1}(f(Z)) \\
\label{e:AX6}& =  ((X- Y)- f(Z)) \cup s_0 \circ\ldots\circ s_{n-1}\left( f\left(\bigcup_{j<n}\rng{h_j}\right)\right) \\
\label{e:AX7}& =  ((X- Y)- f(Z)) \cup \bigcup_{j<n}s_j(f(\rng{h_j}))\\
\label{e:AX7.5}& =  ((X- Y)- f(Z)) \cup \bigcup_{j<n}r_j(q_j(p_j(f(\rng{h_j}))))\\
\label{e:AX7.6}& =  ((X- Y)- f(Z)) \cup \bigcup_{j<n}r_j(q_j(f(\rng{h_j}))))\\
\label{e:AX7.7}& =  ((X- Y)- f(Z)) \cup \bigcup_{j<n}r_j(f^{-1}(f(\rng{h_j}))))\\
\label{e:AX7.8}& =  ((X- Y)- f(Z)) \cup \bigcup_{j<n}r_j(\rng{h_j})\\
\label{e:AX7.9}& =  ((X- Y)- f(Z)) \cup \bigcup_{j<n}\rng{h_j}\\
\label{e:AX9}& =  ((X- Y)- f(Z)) \cup Z 
\end{align}
Here (\ref{e:AX5}) follows by (\ref{e:jfix}), as does (\ref{e:AX7}).  Equations (\ref{e:AX7.6}, \ref{e:AX7.7}, and \ref{e:AX7.8}) follow respectively from (\ref{e:pj1}, \ref{e:qj1}, and~\ref{e:rj1}).

Thus $Z\subset \pi(X- Y)\subset \pi(X)\in [X]$ by (\ref{e:jeqclass});  moreover, (\ref{e:jfix}) implies (\ref{e:Exfixes}).
\end{proof}

\begin{lem}
\label{p:ExttoS}
Suppose there is a $Y\neq X$ with $Y\in [X]
$, and let $Z\subset M- (X\cup Y)$ with $Z \stackrel{f}{\preceq} X\cap Y$.  Suppose further that one of the following obtains:
\begin{align}
\label{e:EtS1}|Z|&\trianglelefteq |X- Y|\\
\label{e:EtS2} |Z|&\trianglelefteq |Y- X|
\end{align}

Then there is a permutation $\pi\in \mathcal M$ such that $[\pi X] = [X]$, $Z\subset \pi(X\cap Y)$, and satisfying (\ref{e:Exfixes}).
\end{lem}
As the idea for the proof of Lemma~\ref{p:ExttoS} is essentially the same as that of Lemma~\ref{p:AtoX}, we omit it.  

Now for two easy lemmata:
\begin{lem}\label{p:ExtoX}
Suppose $Z\subset M- (X\cup Y)$, $Z\stackrel{f}{\preceq}X- Y$ with $Y\in [X]
$ and $Y\neq X$.  Then there is a permutation $\pi\in \mathcal M$ such that $[\pi X] = [X]$, $Z\subset \pi(X- Y)$, satisfying (\ref{e:Exfixes}). 
\end{lem}
\begin{lem}\label{p:AtoS}
Suppose $Z\subset X- Y$, $Z\stackrel{f}{\preceq}X\cap Y$ with $Y\in [X]
$ and $Y\neq X$.  Then there is a permutation $\pi\in \mathcal M$ such that $[\pi X] = [X]$, $Z\subset \pi(X\cap Y)$, and satisfying (\ref{e:Exfixes}).
\end{lem}

\begin{proof}[Proof of Lemmata~\ref{p:ExtoX} and~\ref{p:AtoS}]
For Proposition~\ref{p:ExtoX}, note that $\pi = \iota(f\upharpoonright Z)$ fixes $Y$, so $[X] = [\pi(X)]$ by Proposition~\ref{p:fixes}.  The rest is obvious and routine.

For Lemma~\ref{p:AtoS}, proceed similarly but set $\pi = \iota(f\upharpoonright Z)$.  
\end{proof}

\subsection{Almost complementarity and symmetry}\label{ssec:almC}

The results of the previous section showed what functions are needed in order to, speaking loosely, ``transform one set into another'', while staying in the same equivalence class.  In this section we define properties that determine sufficient conditions for the application of those results.  In particular, we establish Lemma~\ref{lem:class}, which says that as long as these properties obtain for a set of a given cardinality, then all sets of that cardinality are $E$-equivalent.

We begin, then, by defining these properties.

\begin{defn} We say that sets $X,Y$ are \emph{almost complementary in $\mathcal M$} if $|X|=|Y|$ and \emph{none} of the following is satisfied in $\mathcal M$:
\begin{align}
\label{e:almostcapY}
|Y- X|&\trianglelefteq|X \cap Y| \\
\label{e:almostcapX}
|X- Y|&\trianglelefteq|X \cap Y| \\\label{e:almostextY}
|X- Y|&\trianglelefteq |M- (X \cup Y)| \\
\label{e:almostextX}
|Y- X|&\trianglelefteq |M- (X \cup Y)| 
\end{align}

Similarly, we say that $X,Y$ are \emph{symmetric in $\mathcal M$} if $|X|=|Y|$ and for some $Z\in \mathcal M$ distinct from $X$ and $Y$,  $|Z|=|X|$ and neither of the following is satisfied for $\mathcal M$:
\begin{align}
\label{e:zsymm1}
|Z- (X\cup Y)| &\trianglelefteq |X- Y| \\
\label{e:zsymm2}
|Z- (X\cup Y)| &\trianglelefteq |Y- X|
\end{align}
\end{defn}

In many cases, we will establish that either (\ref{e:zsymm1}) or (\ref{e:zsymm2}) holds in $\mathcal M$ by establishing that one of the following holds for $\mathcal M$:
\begin{align}
\label{e:symm1}
|M- (X \cup Y)|  & \trianglelefteq |X- Y| \\
\label{e:symm2}
|M- (X \cup Y)|  & \trianglelefteq |Y- X| 
\end{align}
This is licensed by (\ref{e:nonAtrans}).

The following two propositions follow by (\ref{e:nonAsublp}) and (\ref{e:nonAsubrp}) and \ref{cc}.

\begin{prop}  \label{p:permpres} Suppose $\pi:M\rightarrow M$ is a permutation.  Then:
\begin{enumerate}
\item $X,Y$ are almost complementary if and only if $\pi(X), \pi(Y)$ are almost complementary.

\item $X,Y$ are symmetric if and only if $\pi(X), \pi(Y)$ are symmetric.
\end{enumerate}
\end{prop}

\begin{prop} Distinct sets $X$ and $Y$ cannot be both almost complementary and symmetric.  
\end{prop}

\begin{defn}
In what follows, distinct, $E$-equivalent sets that satisfy one of (\ref{e:almostcapY}--\ref{e:almostextX}) and one of (\ref{e:zsymm1}--\ref{e:zsymm2}) will play a special r\^ole.  Thus we say that sets $X,Y$ are \emph{opportune} if they are distinct, $Y\in [X]$, and they are neither almost complementary nor symmetric.  We moreover will say that a single set $X$ is opportune if there is a $Y$ such that the pair $X,Y$ is opportune.  
\end{defn}

\begin{defn}\label{def:relfin}
Two sets $X, Y$ are \emph{relatively finite in $\mathcal M$} just if $|X- Y| = |Y- X|$ in $\mathcal M$.  (Note that if $X,Y$ are relatively finite, then $X\boxminus Y$.)  We set $RF(X)$ to be the set of sets $Y$ that are relatively finite in $\mathcal M$ to $X$.  
\end{defn}

\begin{lem} \label{lem:class}  If $X,Y$ are opportune 
then $RF(X)\subseteq \xkl$.  That is, every set relatively finite to $X$ in $\mathcal M$ is $E$-equivalent to $X$.
\end{lem}
\begin{proof}  Let $Z\in RF(X)$ 
 with $Z\stackrel{f}\approx X$ in $\mathcal M$; we may assume that $f$ fixes $Z\cap X$ pointwise.  We show that there is a permutation $\pi$ of $M$ such that $\pi(X) = Z$ and $[\pi(X)] = [X]$.

First we partition $Z$:
\begin{align*} 
Z_1 & = f^{-1}(X- Y) \cap Y- X & Z_2 & = f^{-1}(X- Y) - Y \\
Z_3 & = f^{-1}(X\cap Y) \cap Y- X & Z_4 & = f^{-1}(X \cap Y) - Y \\
Z_5 & = f^{-1}(X) \cap X = X \cap Z 
\end{align*} 
where the equality in $Z_5$ is due to the choice of $f$.  It is easy to see that these sets are pairwise disjoint and that their union is $Z$.

Note that $Z_1\stackrel{f\upharpoonright Z_1}\preceq X- Y$.  Since $X$ and $Y$ are not almost complementary,  at least one these holds: 
\begin{align}
|Z_1|&\trianglelefteq|X\cap Y| &
\label{e:LAXExt}
|Z_1|&\trianglelefteq|M- (X\cup Y)|
\end{align}
Either way Lemma~\ref{p:AtoX} provides a permutation $\pi_1\in \mathcal M$ such that 
$[\pi_1 X]=[X]$ and $Z_1\subset \pi_1(X- Y)\subset \pi_1(X)$ and satisfying (\ref{e:Exfixes}).  Note that by (\ref{e:Exfixes}), $Z_5 \subset \pi_1(X)$, so 
\begin{equation}
\label{e:Z1Z5} Z_1\cup Z_5\subset \pi(X)
\end{equation}
Set $X_1 = \pi_1(X), Y_1= \pi_1(Y)$.    

Then since $Z_2\subset M- (X_1\cup Y_1)$ (by (\ref{e:Exfixes})), and since (by (\ref{e:Exfixes}) again) $Z_2\stackrel{f\upharpoonright Z_2}\preceq X_1- Y_1$, by Proposition~\ref{p:ExtoX} there is a permutation $\pi_2\in \mathcal M$ with $Z_2 \subset \pi_2(X_1- Y_1) \subset \pi_2(X_1) \in [X]$ and satisfying (\ref{e:Exfixes}).  By (\ref{e:Z1Z5}) and (\ref{e:Exfixes}), we have 
\begin{equation}
\label{e:Z2}
Z_1\cup Z_2 \cup Z_5 \subset \pi_2(X_1) \in [X]
\end{equation}

Set $X_2 = \pi_2(X_1), Y_2= \pi_2(Y_1)$.  Since $Z_3\subset Y_2- X_2$  and  $Z_3 \stackrel{f\upharpoonright Z_3}\preceq X_2\cap Y_2$ (by (\ref{e:Exfixes}) twice), there is by Lemma~\ref{p:AtoS} a permutation $\pi_3\in \mathcal M$ satisfying (\ref{e:Exfixes}) and with 
$Z\subset \pi_3(X_2 - Y_2) \subset \pi_3(X_2) \in [X]$.  By (\ref{e:Z2}) and (\ref{e:Exfixes}) we have that 
\begin{equation}
\label{e:Z3}
Z_1\cup Z_2\cup Z_3\cup Z_5 \subset \pi_3(X_2)
\end{equation}

Finally set $X_3 = \pi_3(X_2), Y_3 = \pi_3(Y_2)$.  Note that $Z_4 \subset M- (X_3\cup Y_3)$ and  that $Z_4 \stackrel{f\upharpoonright Z_4}\preceq X_3\cap Y_3$ (by (\ref{e:Exfixes}) twice).  Notice now  that
since $\pi_1, \pi_2$, and $\pi_3\in \mathcal M$ are permutations and since $X,Y$ are not symmetric, $X_3, Y_3$ are not symmetric by Proposition~\ref{p:permpres}.  Thus  we have at least one of 
\begin{align}
\label{e:LEtS1} |Z_4| &\trianglelefteq |X_3- Y_3| & |Z_4| &\trianglelefteq |Y_3- X_3| 
\end{align}
obtains.
So by Lemma~\ref{p:ExttoS}, there is a permutation $\pi_4\in \mathcal M$ such that 
\begin{equation}
\label{e:Z_4inX}
Z_4 \subset \pi_4(X_3\cap Y_3) \subset \pi_4(X_3) \in [X]
\end{equation}  
By (\ref{e:Z3}) and again (\ref{e:Exfixes}) we have that 
\begin{equation}
Z= Z_1\cup Z_2\cup Z_3\cup Z_4 \cup Z_5 = \pi_4(X_3) = \pi_4(\pi_3(\pi_2(\pi_1(X)))) = \pi(X)
\end{equation}
for  $\pi = \pi_1 \circ \pi_2\circ\pi_3\circ \pi_4$; which was to be established.
\end{proof}

Dealing with opportune sets via Lemma~\ref{lem:class} will enable us to deal with cases in which $X$ is finite, both within and outside of $\mathcal M$.  We use a special class of opportune sets to deal with $X$ infinite in $\mathcal M$.

\begin{defn}
We will say that $X,Y$ are \emph{ideally opportune} if they are opportune and 
either $|X|\trianglelefteq |M- X| 
$ or $|Y|\trianglelefteq |M- Y| 
$.  We will also say that the single concept $X$ is ideally opportune if there is a $Y$ such that $X,Y$ are opportune and $|X|\trianglelefteq |M- X|$.
\end{defn}


We now turn to proving another lemma that will address the more general case, in which $X$ is not finite in the metatheory (though it may be finite in $\mathcal M$).
We begin with another preparatory lemma.

\begin{lem}\label{lem:extend}
Suppose $X$ is ideally opportune, and $X \boxminus Z$.  Then $Z\in [X]$.
\end{lem}
\begin{proof}
Note that if $|X|<\omega$ then $X$ and $Z$ are relatively finite, so 
$Z\in [X]_{E(\boxminus)}$ by Lemma~\ref{lem:class}.  So we assume $|X|\geq\omega$.

Thus by Proposition~\ref{p:triinf}, and since $X$ is ideally opportune, we may assume that $|Z|=|X|\leq |M-X|=|M-Z|$.

Now if $|X|=|M-X|$, then  
\[ |M- Z|\geq |Z|=|X|=|M- X|=|M|\]
If $|Z- X|<|X- Z|$, then as $M- X = (M- (X\cup Z))\cup (Z- X)$, 
by \ref{ism}, we  obtain
\begin{equation}
\label{e:M-XZ}
|X|\leq |M- (X\cup Z)|
\end{equation}  
Similarly for $|X- Z|<|Z- X|$.

On the other hand if  $|X|<|M- X|$, then as $|X|=|Z|$, $|X\cup Z|<|M- X|$ by \ref{ism}.  Then by \ref{ism} again, we have that 
\[|M- (X\cup Z)|=|M|\geq |M- X|>|X|\] which again yields (\ref{e:M-XZ}).  

Given (\ref{e:M-XZ}) with $f$ a witnessing injection, observe that $f(X)$ and $X$ are disjoint and so relatively finite.  By Lemma~\ref{lem:class}, since $X$ is one of a pair of opportune sets, $f(X)\in [X]_{E(\boxminus)}$.  But then $f(X), X$ are opportune as well.  Moreover, as $Z, f(X)$ are disjoint and equinumerous, they also are relatively finite.  So by Lemma~\ref{lem:class}, $Z\in [f(X)]_{E(\boxminus)} = [X]_{E(\boxminus)}$.
\end{proof}

\subsection{The intermediate Theorem~\ref{thm:main-0formal}} \label{sec:classif}

We are now nearly ready to prove Theorem~\ref{thm:main-0formal}. 
We first deal with degenerate cases, namely those in $E(\boxminus)_X$ that may be trivial \emph{and also} separative.  
The following propositions follow from \ref{ipm} and (\ref{e:nonApsplit}) and (\ref{e:nonAexp})  of Proposition~\ref{p:nonA}.

\begin{prop}
If \begin{equation}
\label{e:propvsmall}
\mathcal M\models |M|\leq 1 \vee |X|=0 \vee |X|=|M|< \omega
\end{equation}  
then $E(\boxminus)_X$ is both a separation and trivial.

Further, if $|M|=2$ and $|X|=1$, then $E(\boxminus)_X$ is either both trivial and a complementation, or it is a separation.
\end{prop}


\begin{prop}\label{p:symsep}
If $Y,Z\in [X]_\boxminus$ are symmetric, then $|Y\triangle Z|\vartriangleleft |X|$.\end{prop}
\begin{prop}
\label{p:almcompcompl}  \sloppy
If $Y,Z\in [X]_\boxminus$ are almost complementary, then $|Y\cap Z|, |M-(Y\cup Z)|\vartriangleleft |X|$.
\end{prop}
\fussy

\begin{prop}\label{p:inf-exclusive}
If $\mathcal M\models \omega \leq |X|\leq |M-X|$ and $E(\boxminus)_X$ is either separative or complementative, then it is not trivial.
\end{prop}


The crucial move in proving Theorem~\ref{thm:main-0formal} is in the following Lemma:
\begin{lem}\label{lem:oneside}
Suppose $\mathcal M\models  |M|>2 \wedge  1<|X|\leq|M-X|$ and $E$ is an $L_0$ definable equivalence relation over $\mathcal M$.  Then $E(\boxminus)_X$ is exactly one of:  trivial, properly separative, or properly complementative.\end{lem}
\begin{proof}
We first establish the lemma for the case where $|X|<\omega$ in $\mathcal M$.  
Obviously no more than one can obtain; we show at least one must.
Since all equinumerous sets finite in $\mathcal M$ are relatively finite in $\mathcal M$, by Lemma~\ref{lem:class}, if there are opportune sets in $[X]_{E(\boxminus)}$, then 
\[[X]_\boxminus = [X]_\boxminus \cap RF(X) \subseteq [X]_{E(\boxminus)}\]
and so $E(\boxminus)_X$ is trivial.  So we may assume there are no opportune sets $Y,Z \in \xkl$, and thus 
that for any $Y,Z\in \xkl$, either $Y, Z$ are symmetric, or $Y, Z$ are almost complementary by Lemma~\ref{lem:class}.

Suppose now that all pairs of set in $[X]_{E(\boxminus)}$ are symmetric.  
By Proposition~\ref{p:symsep}, all such pairs satisfy (\ref{e:sepv}).  Now by Proposition~\ref{p:permpres} this holds for every concept bicardinally equivalent to $X$; thus $E(\boxminus)_X$ is properly separative.

On the other hand, if there are $Y,Z\in \xkl$ 
that are not symmetric, then these are almost complementary.  By Proposition~\ref{p:almcompcompl}, the non-symmetric concepts in $\xkl$ satisfy (\ref{e:compv}).  Clearly then (\ref{e:sepv}) fails of them, so $E(\boxminus)_X$ is not separative.    By Proposition~\ref{p:symsep} all the symmetric concepts satisfy (\ref{e:sepv}), and so (\ref{e:compv}).  Proposition~\ref{p:permpres} ensures then that $E(\boxminus)_X$ is complementative, and since (as just shown) not separative, it is properly complementative.    So the result holds for all $|X|$ finite in $\mathcal M$.

The case in which $X$ is infinite in $\mathcal M$ is similar.  By Proposition~\ref{p:inf-exclusive}, at most one of the options can hold; we show at least one must.  By Lemma~\ref{lem:extend} if $E(\boxminus)_X$ is non-trivial then there are no ideally opportune concepts in $\xkl$.  Since by assumption $|X|\leq |M-X|$ it follows that there are no opportune concepts in $\xkl$.  The above argument establishes again that $E(\boxminus)_X$ is either properly separative or properly complementative. 
\end{proof}

\begin{cor}
\label{cor:anti1}
Suppose $\mathcal M\models  |M|>2 \wedge  1<|M-X|<|X|$ and $E$ is an $L_0$ definable equivalence relation over $\mathcal M$.  Then $E(\boxminus)_X$ is exactly one of:  trivial, properly separative, or properly complementative.
\end{cor}
\begin{proof}
It is easy to see that if $E$ is an equivalence relation, then so is the $L_0$ definable
\[ E^c(X,Y) \Leftrightarrow E(M-X, M-Y) \]
Thus, by Lemma~\ref{lem:oneside}, $E^c(\boxminus)_{M-X}$ is exactly one of trivial, properly separative, or properly complementative.  But by the definition of $E^c$, $E^c(\boxminus)_{M-X}$ has the same profile as $E(\boxminus)_X$.
\end{proof}

Lemma~\ref{lem:oneside} and Corollary~\ref{cor:anti1} suffice to prove the first part of Theorem~\ref{thm:main-0formal}.  The remainder follows from the following lemma.

\begin{lem}\label{p:nuisancelike}
Suppose $\mathcal M\models |X|<\omega$ and there is a $W\neq X$ with $W\in [X]_{E(\boxminus)}$, and $E(\boxminus)_X$ non-trivial.  Then for any $Y,Z \in [X]_{\boxminus}$, either 
(\ref{e:septight}) or (\ref{e:comptight}) holds.  
\end{lem}

\begin{proof}
By the first part of Theorem~\ref{thm:main-0formal}, $E(\boxminus)_X$ is either properly separative or  properly complementative. By Lemma~\ref{lem:class}, we know that though $X,W$ are relatively finite (since they are finite), they are not opportune, so $X,W$ are either  symmetric or  almost complementary. 

We will first show that if $X,W$ are symmetric, then for any $Z\in [X]_\boxminus$, 
\begin{equation}\label{e:septightt}  |X\triangle Z| \vartriangleleft |X| \Rightarrow \mathcal M\models E(X,Z)\end{equation}

Assume $X,W$  symmetric, so for some $H$ with $|X|=|W|=|H|$, 
\begin{equation}\label{e:H-X}
|H-(X\cup W)|\vartriangleright |X-W|, |W-X|\end{equation}
  Assuming 
\begin{equation}\label{e:sepassmpZ} 
|X\triangle Z| \vartriangleleft |X|\end{equation}
we have that $|X|=|H|\vartriangleright |H-(X\cup W)|$; it then follows from (\ref{e:nonApslcor}) of Proposition~\ref{p:nonA} that 
\begin{equation}
\label{e:nsfineq1}  |X\cap W| \vartriangleright |X-W|, |W-X|
\end{equation}
So by (\ref{e:nonAsplit}) of Proposition~\ref{p:nonA}, 
we have  \begin{equation}
\label{e:Zjustright}
|Z-  X|\leq |X\triangle Z| \vartriangleleft |X\cap W|
\end{equation}
Thus by Lemma~\ref{p:AtoX}, $Z-X\subseteq \pi(X)$
 with $\pi(X\cap Z)=X\cap Z$ and $E(\pi(X), X)$.  Thus $Z\in [X]_E$, and this establishes (\ref{e:septight}).

Now we will show that if $X,W$ are almost complementary, then for any $Z\in [X]_\boxminus$, 
\begin{equation}\label{e:compltight}  \left[ |X\triangle Z| \vartriangleleft |X| \; \textrm{or} \; |X\cap Z|,|M-  (X\cup Z)|\vartriangleleft |X|\right]  \Rightarrow \mathcal M\models E(X,Z)\end{equation}
So assume $X,W$ are almost complementary.  

If (\ref{e:sepassmpZ}) obtains, then we have 
\begin{equation}
\label{e:Zcomplright}
|(Z-  X)\cap W| \leq |Z-  X|\leq |Z\triangle X| \vartriangleleft |X\cap W|
\end{equation}
by (\ref{e:nonApslcor}) of Proposition~\ref{p:nonA}, and so $Z\in [X]_E$ by Lemma~\ref{p:AtoX}.

On the other hand suppose
\begin{align}
\label{e:othercomplright}
|X\cap Z|&\vartriangleleft |X| \\
\label{e:othercomplright1} |M-  (X\cup Z)|&\vartriangleleft |X|
\end{align}
By (\ref{e:othercomplright}) and (\ref{e:nonApslcor}), $|X\cap Z|\vartriangleleft |X-Z|$.  But then by (\ref{e:othercomplright1}) and (\ref{e:nonApslcor}) again,\footnote{The reader may wish to aid her reasoning in these following arguments using Venn Diagrams;  label a set $A$  ``S'' for \emph{relatively small}, and $B$ ``L'' for \emph{relatively large}, if $A\triangleleft B$.  Proposition~\ref{p:nonA} basically asserts that the expected reasoning will obtain (e.g., if $A$ is relatively large but $A\cap B$ is relatively small, it follows that $A-B$ is relatively large).}
\begin{align}
|(W-  Z) \cap X|  \leq |X\cap Z| &\vartriangleleft |W\cap Z| & \textrm{and}\\
|W-  (X\cup Z)| \leq |M-  (X\cup Z)| & \vartriangleleft |W\cap Z|,  & \textrm{so} \\
|W-  Z| & \vartriangleleft |W\cap Z| & 
\end{align}
by (\ref{e:nonApsplit}).
As similar argument shows that 
\begin{equation}
|Z-  W| \vartriangleleft |W\cap Z|
\end{equation}
and thus that $|W\triangle Z|\vartriangleleft |W|$ by (\ref{e:nonAsplit}).
From here we can use the argument from earlier in this proof establishing (\ref{e:septight}), and the fact that $[W]_E = [X]_E$, and we have that $Z\in [X]_E$.
\end{proof}

From Theorem~\ref{thm:main-0formal} we obtain the following corollary, which we will use in proving Theorem~\ref{thm:main-formal}.

\begin{cor}
\label{cor:finiteMT}  If $X$ is finite in the metatheory, then $[X]_{E(\boxminus)}$ is either $[X]_\boxminus$, or $\{X\}$, or $\{X, M-X\}$.
\end{cor}
\begin{proof}
If $X$ is finite in the metatheory then $|X\triangle Y|\vartriangleleft |X|$ if and only if $|X\triangle Y|=0$.  Likewise for $|X\cap Y|, |M-(X\cup Y)|\vartriangleleft |X|$.
\end{proof}

\subsection{The Classification Theorem~\ref{thm:main-formal}}\label{ssec:2ndclthm}

To prove  Theorem~\ref{thm:main-formal} we need to use \ref{cwf} to address some  complications lingering due to our use of non-standard semantics.
We show that under \ref{cwf} all proper separations and proper complementations  on finite concepts are the finest possible, even if ``finite'' only means ``finite in $\mathcal M$''.   Theorem~\ref{thm:main-formal} follows from the theorem proved in this section.  

\begin{defn}
    Suppose that $E$ is properly separative on $[X]_{\boxminus}$.  Then $J\in \mathcal M$ is called a \emph{measure} of $E$ on $[X]_\boxminus$ just if for all $Y, Z\in [X]_\boxminus$, 
    \[ E(Y, Z) \Leftrightarrow |Y\triangle Z| \vartriangleleft |J| \]
    
    Suppose that $E$ is properly complementative on $[X]_{\boxminus}$.  Then $J\in \mathcal M$ is called a \emph{measure} of $E$ on $[X]_\boxminus$ just if for all $Y, Z\in [X]_\boxminus$, 
    \[ E(Y, Z) \Leftrightarrow |Y\triangle Z| \vartriangleleft |J| \; \textrm{or} \; |M-(Y\triangle Z)| \vartriangleleft |J|\]
    
\end{defn}

\begin{thm}\label{lem:nonstandard}
    Let $\mathcal M \models |X|<\omega$.  If $E$ is nontrivial at $[X]_\boxminus$ then its measure is $\emptyset$.  
\end{thm}

\begin{proof}
    If $E$ is nontrivial on $[X]_{\boxminus}$ then by Theorem~\ref{thm:main-0formal} it is either a proper separation or a proper complementation.  Suppose $J$ is a non-zero measure of $E$ on $[X]_{\boxminus}$.

    Let $L(\xi)$ abbreviate the formula 
    \[ (\forall U, V\in [X]_\boxminus)(E(U,V) \rightarrow ( |U\triangle V|< |\xi| \vee |M-(U\triangle V)|<|\xi|)\]
    
    Clearly $L(J)$, so by \ref{cwf}, there is a smallest $J'$ such that $L(J')$.  If $J'\neq \emptyset$, then observe that by Corollary~\ref{cor:finiteMT} that $|J'|\neq n$ for any natural number $n$---that is, $J$ is not finite in the metatheory.  Thus, letting $j\in J'$  we have that $\neg L(J'-\{j\})$, and so there are $U, V\in [X]_{E(\boxminus)}$ such that 
    \begin{equation}
    \label{e:t6.1}  E(U,V) \; \textrm{but}\; |U\triangle V| \not < |J'-\{j\}|
    \end{equation} (If the witnessing $U,V$ to $\neg L(J')$ satisfy the second disjunct of the consequent, just choose $U, M-V$ so that they satisfy \eqref{e:t6.1}).  As $L(J)$ we have that $|U\triangle V| = |J'-\{j\}|$.
    
    By assumption $E(\boxminus)_X$ is non-trivial, so regardless of whether it is a complementation or a separation, there is a $Z$ with 
    \begin{align}
    U& \neq V & |U\triangle V|&\vartriangleleft |J| & |Z-(U\cup V)| \trianglelefteq |J|  \label{e:t6}
    \end{align}
    As $|X|\geq |J|$, $|U|=|V|=|X|$, and $|U\triangle V|\vartriangleleft |J|$, we have that $|U\cap V|\trianglerighteq |X|$ by Proposition~\ref{p:nonA}.  Thus, with \eqref{e:t6} we have that 
    \begin{equation}
    |U\cap V|, |Z - (U\cup V)| >2
    \end{equation}
    Thus, let $a \in U\cap V$ and $b, c\in Z- (U\cup V)$, and set 
    \begin{align}
    U' &= (U-\{a\})\cup \{b\}   & V' & = (V-\{a\})\cup \{c\} 
    \end{align}
    Note now that as 
    \begin{equation}
    |U\triangle U'| = |V \triangle V'| = 2  \vartriangleleft |J|
    \end{equation}
    we have $E(U, U'), E(V, V')$, and so $E(U', V')$.  Thus, since $L(J)$, \mbox{$|U'\triangle V'| < |J|$}, but it is easy to see that the construction of $U'$ and $V'$ ensure that $|U'\triangle V'| = |J\cup \{a\}| > |J|$
\end{proof}

Theorem~\ref{lem:nonstandard} proves Theorem~\ref{thm:main-formal} for $X$ finite in $\mathcal M$.  For the rest one needs only consider $X$ such that $\mathcal M\models |X|\geq \omega$.  By Proposition~\ref{p:triinf}, if $E(\boxminus)_X$ is separative, then

\begin{equation}
\label{e:ifsep}
\mathcal M\models (\exists X'\subseteq X)(\forall Y,Z\in [X]_\boxminus)(E(Y,Z) \rightarrow |Y\triangle Z|<|X'|) \end{equation}
and likewise if $E(\boxminus)_X$ is complementative.  This completes the proof of Theorem~\ref{thm:main-formal}.

\section{The classification of ``Bad Companions''}\label{sec:badco}  \label{sec:neol}

The ``bad company problem'' faced by neo-logicists is the two part challenge of finding a plausuble criterion for the ``logicality'' of abstraction principles, and showing that the good abstraction principles like {\tt HP} are logical in this sense, while the bad ones like {\tt BLV}, {\tt NP}, and {\tt CP} are not.  It is thus relevant to note how Theorem~\ref{thm:main-formal} and its apparatus bear on the joint consistency of abstraction principles, since at the very least any abstraction principles qualifying as ``logical'' should be jointly consistent.  

Consistency results in the absence of well-behaved cardinalities are difficult to obtain.  As such our discussion will be directed towards the neo-logicist who thinks 
``logical'' abstraction principles must be jointly consistent \emph{in the presence of well-behaved cardinalities}.  We'll thus use ``jointly consistent'' as if
 it implicitly has the qualification on the behavior of cardinalities.  

Such a neo-logicist is committed to thinking that joint consistency with {\tt HP} is a necessary condition for the logicality of abstraction principles.  In the remainder of this section we apply Theorem~\ref{thm:main-formal} and its apparatus to discern which abstraction principles are consistent with a Dedekind infinite universe.  These obviously determine which abstraction principles are consistent with {\tt HP}, and thus which must be ruled out from being ``logical''. 

One reason why the bad company problem can appear vexing is that any ``logical'' equivalence relation can give rise to an abstraction principle.  Facing an untamed menagerie of such equivalence relations, the neo-logicist would face an equally wild zoo of abstraction principles.  The import of Theorem~\ref{thm:main-formal} is in taming the menagerie---or better, classifying the species.  At each bicardinal slice of $\mathcal M$, there are only three ``classes'' of equivalence relations on that slice.  This enables us to see clearly why certain abstraction principles have the types of models they do:  at certain (finite or infinite!) bicardinal classes, \emph{non-trivial} equivalence relations distinguish more equivalence classes than there are objects.  

We have said that {\tt BLV}, {\tt NP}, and {\tt CP} all count as bad companions.  We thus have so far two bad companions deploying refinements of separations, and one deploying a refinement of a complementation.  Of the former, {\tt BLV} involves the finer (in fact finest) separation;  as we will see in contexts with \ref{ipm} it is in a sense the paradigm  case of a separative abstraction principle.  We will also discuss the paradigm case of a complementative abstraction principle;  it is not {\tt CP} but a close relative, the \emph{Liberated Complementation Principle} ({\tt LCP}):
\[ (\forall X, Y)(\ell X = \ell Y \leftrightarrow X=Y \vee X = M-Y)\]

\begin{rmk}
The equivalence relation for {\tt BLV} is the finest equivalence relation that refines a separation for every $X\in \mathcal M$ (it is also the finest equivalence relation full stop, see~\cite{Antonelli2010aa}).

The equivalence relation for {\tt NP} refines a separation at all bicardinal slices of $\mathcal M$ for Dedekind infinite sets.  

Finally, the equivalence relation for {\tt LCP} refines a complementation at $X\in \mathcal M$ if $|X|=|M-X|$, and in fact it is the finest equivalence relation refining a complementation at bicardinal slices for such $X$.
\end{rmk}

\subsection{Translation and restricted abstraction principles}\label{ssec:trans}

In the remainder of this section we will need the following definitions:
\begin{defn}
For $\mathcal M \models |X|, \omega \leq |M|=|M-X|$, we will write $Y\leq [X]_\boxminus$ as shorthand for $|Y|<|X| \vee Y\in [X]_\boxminus$.

Let $E$ be an equivalence relation.  Let $\varphi(Y,Z)$ be a formula.  We say that $\varphi$ \emph{is functional below $X$ in $\mathcal M$} if 
\begin{equation}\label{e:interp-uniqdn} \mathcal M\models |X|\leq |M-X| \wedge (\forall Y)(Y\leq [X]_\boxminus \rightarrow 
(\exists! Z)\varphi(Y,Z))\end{equation}
We say that $\varphi$ is functional \emph{at $X$} if 
\begin{equation}\label{e:interp-uniq} \mathcal M\models (\forall Y)(Y\in [X]_\boxminus \rightarrow 
(\exists! Z)\varphi(Y,Z)\end{equation}

Often in the presence of such a $\varphi$ we will use functional notation, using $F_\varphi(Y)$ for the unique object $Z$ such that $\varphi(Y,Z)$, 
dropping the subscript where it is unneeded for understanding.

If $A_E[\partial]$ is an abstraction principle, and $\varphi$ a formula functional below $X$ in $\mathcal M$, then we say that $\theta$ \emph{translates} $A_{E'}(\leeq X)$ \emph{via} $\varphi$ in $\mathcal M$ to mean that 
\begin{equation} \label{e:AEuptoXU} \mathcal M \models (\forall Y, Z)(Y,Z\leq  [X]_\boxminus  \rightarrow 
(\partial(F_\varphi(Y)) = (\partial(F_\varphi(Z)) \leftrightarrow E'(Y,Z)))\end{equation}
If $\varphi$ is functional at $X$, then we say that $A_E[\partial]$ translates $A_{E'}(\eeq X)$ via $\varphi$ in $\mathcal M$  to mean that 
\begin{multline} \label{e:AE=XU} \mathcal M \models (\forall Y, Z)(Y,Z\in [X]_\boxminus \rightarrow  (\partial(F_\varphi(Y)) = \partial(F_\varphi (Z)) \leftrightarrow E'(Y,Z)))\end{multline}
Finally we may say that $A_E[\partial]$ translates $A_{E'}$ in $\mathcal M$ via $\varphi$ just if 
\begin{equation}
\mathcal M \models (\forall Y,Z)(\partial (F_\varphi(Y))=\partial(F_\varphi(Z)) \leftrightarrow E'(Y,Z) )
\end{equation}

These definitions afford us the capacity to talk of the \emph{restricted abstraction principles}\label{p:restap}
$A_E[\partial](\leeq X)$ and $A_E[\partial](\eeq X)$ 
in the sense that $\mathcal M$ satisfies one of these just if ``$(\forall x)(x=x)$'' translates that principle in $\mathcal M$ via ``$Y=Z$''.    
\end{defn}

In what follows of this section we show that if $2<|X|=|M-X|$ and $E(\boxminus)_X$ is non-trivial, then (in the presence of \ref{ipm}) $A_E$
is inconsistent.  The engine of these results will be the following two propositions.

\begin{prop}\label{p:interp-incons}
Let $\varpi$ be an $L_0$-sentence. Suppose that for all $\mathcal M\models A_E[\partial]$, $A_E[\partial] \wedge \varpi$ translates $A_{E'}(\eeq M)$ in $\mathcal M$ via $\varphi$, and that  $A_{E'}(\eeq M)$ and $\varpi$ are jointly inconsistent (in our background logic).  Then $A_E$ and $\varpi$ are  jointly inconsistent (in our background logic) as well.
\end{prop}
\begin{proof}
If $A_E\wedge \varpi$ is consistent then by the completeness of second-order logic for the non-standard semantics, there is a structure $\mathcal M = (M, S_1[M], \ldots, \partial) \models A_E \wedge \varpi$.  As $\varphi(Y,Z)$ is functional at $M$, define $\partial' Y = \partial (F_\varphi(Y))$;  $\partial'$ is then a function in $\mathcal M$ from concepts to objects.  We now augment $\mathcal M$ to $\mathcal M' = (M, S_1[M], \ldots, \partial, \partial')$, and note that since $A_E$ translates $A_{E'}(\eeq M)$, $\mathcal M' \models A_{E'}[\partial'](\eeq M)$.  Since $\varpi$ is an $L_0$-sentence (with parameters from $\mathcal M$), $\mathcal M' \models \varpi$.  A routine induction on the complexity of formulas shows that $\mathcal M'$ satisfies the comprehension axioms in the expanded language $L_0[\partial']$. 
Thus if $A_E[\partial]\wedge \varpi$ has a model then so does $A_{E'}[\partial'](\eeq M)\wedge \varpi$.
\end{proof}

For the next proposition, observe that if $\mathcal M\models |M|\geq \omega$, then \ref{ipm} ensures that there is a bijection $\langle \cdot, \cdot \rangle : M\times M\rightarrow M$.  We will use this notation in what follows, as well as writing $\langle X,Y\rangle$ to mean $\{ \langle x, y\rangle \mid x\in X, y\in Y\}$.\footnote{The arguments of this section elaborate on the one given in~\cite{ED-2015}.}  
We also obtain:
\begin{prop}\label{p:pairingimage}
For $|M|\geq \omega$, 
let $\langle \cdot, \cdot \rangle:M\times M \rightarrow M$ be a bijection, whose existence is assured by \ref{ipm}.   For $X,Y$, if either 
\begin{equation}
|Y|=1 \leq |X| \leq |M-X|
\end{equation} or 
\begin{equation}
|Y|,\omega \leq |X|\leq |M-X|
\end{equation}
then $\langle Y,X\rangle \in [X]_\boxminus$. 
\end{prop}

\begin{lem}\label{lem:refinterp}
Let $\mathcal M \models |X|\geq \omega$.  If $\mathcal M \models |X|\leq |M-X|$ and $E(\boxminus)_X$ refines a separation, then $A_E$ translates ${\tt BLV}(\leeq X)$ via  
\begin{equation}\label{e:translation}
\varrho(Y,Z) := Z = \langle Y, X\rangle
\end{equation}

Under the same conditions, if $E(\boxminus)_X$ refines a complementation (and does not refine a separation),
 then $A_E$ translates ${\tt LCP}$ via 
\begin{multline}\label{e:translation2}
\varphi(Y,U) := (\exists V)((
(\partial\emptyset \in Y \wedge V = \langle Y, X\rangle )\vee 
(\partial\emptyset \in M-Y \wedge V = \langle M-Y , X\rangle))\\
 \wedge ((V\neq \langle M, X\rangle \wedge V=U) \vee (V=\langle M, X\rangle \wedge U=\emptyset)))
\end{multline}
\end{lem}

\begin{proof}
For both assertions, note:
By hypothesis, $\mathcal M\models |M|\geq |X|\geq\omega$, so by \ref{ipm} the bijection $\langle \cdot, \cdot \rangle$ exists. Notice that since $E(\boxminus)_X$ is non-trivial, $\partial\emptyset \neq \partial Y$ for any $Y\in [X]_\boxminus$, since otherwise by \ref{p:perminv} $E(\boxminus)_X$ would be trivial, as $E(\emptyset, X)$ holds whenever $E(\emptyset, f(X))$ holds for $|X|=|f(X)|$ and $|M-X|=|f(M-X)|$. Further if this 
$Y\neq \emptyset$ then $\langle Y, X\rangle \in [X]_\boxminus$ by Proposition~\ref{p:pairingimage} (under the conditions of the second assertion this follows since  $|X|=|M-X|=|M|$.)

We prove the first assertion in brief:  It is easy to verify that $\varrho$ is functional (and so is functional below $X$).  Let $|Y|, |Z|\leq |X|$.  

For distinct $Y,Z\neq \emptyset$,
$|\langle Y, X\rangle \triangle \langle Z, X\rangle | = |X|$ by \ref{ipm}, so if $Y\neq Z$ then since $E(\boxminus)_X$ is a separation, $\neg E(\langle Y, X\rangle ,\langle Z, X \rangle )$, and so by $A_E$, $ 
\partial \langle Y, X\rangle  \neq \partial \langle Z, X\rangle 
$.  Conversely if $Y=Z$ then $\langle Y, X\rangle = \langle Z, X \rangle $.  Since $E$ is an equivalence relation, it follows that $E(\langle Y, X\rangle, \langle Z, X\rangle  )$, and so, by $A_E$, that $\partial\langle Y,X\rangle  = \partial\langle Z,X\rangle$.

For the second assertion, 
it is easy to verify that $\varphi$ is functional (and so is functional below $X$).  Towards verifying the consequent of (\ref{e:AEuptoXU}) and working in $\mathcal M$, we need only show that 
\begin{equation}\label{e:LCPtrans} \partial U_\varphi (Y) = \partial U_\varphi (Z) \leftrightarrow Y=Z \vee Y=M-Z \end{equation}
for all concepts $Y,Z \in \mathcal M$.

By the definition of $U_\varphi$, $U_\varphi(\emptyset) = U_\varphi(M) = \emptyset$.  So clearly if $X$ and $Y$ are chosen from $\emptyset, M$, then (\ref{e:LCPtrans}) obtains.

Now, if $Y=\emptyset$ then $U_\varphi(Y)=\emptyset$, and if $Y=M$ then $U_\varphi(Y) = \emptyset$.  On the other hand if $Y\neq \emptyset, M$, then $U_\varphi(Y) \in [X]_\boxminus$. 
Thus by $A_E[\partial]$, $\partial (U_\varphi(Y)) \neq \partial(U_\varphi(\emptyset))=\partial(U_\varphi(M))$.  Thus:
\begin{gather}
\emptyset \neq Z \rightarrow \partial(U_\varphi \emptyset) \neq \partial(U_\varphi Z)\\
M \neq Z \rightarrow \partial (U_\varphi M) \neq \partial(U_\varphi Z)
\end{gather}
Now if $Y,Z\neq \emptyset, M$ then we have 
\begin{gather}
Y\neq Z \rightarrow |U_\varphi(Y)\triangle U_\varphi(Z)| = |X| \\
Y \neq M-Z \rightarrow |M-(U_\varphi(Y)\triangle U_\varphi(Z))| = |X|
\end{gather}  

So, since $E(\boxminus)_X$ refines a complementation we have that  
\begin{equation}\label{e:translate-necc}
\partial(U_\varphi(Y)) = \partial(U_\varphi (Z)) \rightarrow Y=Z \vee Y = M-Z
\end{equation}
by $A_E[\partial]$.

Clearly by the construction of $U_\varphi$: 
\begin{gather}
Y=Z \rightarrow E(U_\varphi(Y),U_\varphi(Z)) \\
Z\neq \emptyset \neq Y= M-Z \rightarrow E(U_\varphi(Y), U_\varphi(Z)) 
\end{gather}
So we have, again by $A_E[\partial]$ and the $E(\boxminus)_X$ refining a complementation:
\begin{equation}
\label{e:translate-suff}
Y = Z \vee Y = M-Z \rightarrow \partial(U_\varphi(Y)) = \partial(U_\varphi (Z))
\end{equation}
Which completes the proof of the translation of 
${\tt LCP}$.  
\end{proof}

We know now some of which abstraction principles translate restrictions of {\tt BLV} and {\tt LCP}.  We now show when, in the presence of \ref{ipm}, these restrictions are inconsistent for each such principle.  Each of the following two subsections will conclude by applying Lemma~\ref{lem:refinterp};  the final subsection applies Theorem~\ref{thm:main-formal} to unify those results.

\subsection{Refinements of separations and ${\tt BLV}(\eeq X)$}

The utility of Lemma~\ref{lem:refinterp} is in relating equivalence relations that are non-trivial on bicardinal slices to what might be called their prime examples.  In this section we treat {\tt BLV} and its restrictions as the prime examples of separations, and show that in infinite structures, refinements of separations translate certain restrictions of ${\tt BLV}$.  The next section proves analogous results for {\tt LCP} and its restrictions.  

From the Lemma we obtain very quickly:

\begin{cor}\label{cor:BLv-inf}
Suppose $\mathcal M\models |X|= |M-X|=|M|$.  Then $\mathcal M\not \models {\tt BLV}(\eeq X)$.  
\end{cor}
\begin{proof}
Given $X$ as in the hypothesis, if $M \models {\tt BLV}(\eeq X)$ then $E(\boxminus)_X$ is a separation, and so refines a separation.  Thus ${\tt BLV}(\eeq X)$ translates both ${\tt BLV}[\varepsilon_1](\leeq X)$ and ${\tt BLV}[\varepsilon_2](\leeq M-X)$.  With $M\stackrel{f}{\approx}X$ and $M\stackrel{g}{\approx}M-X$, set 
\[ \varepsilon Y = \begin{cases}
					f(\varepsilon_1 Y) & |Y|\leq X \\
					g(\varepsilon_2 Y) & |M-Y|<|M-X|
\end{cases}\]
Clearly then ${\tt BLV}(\eeq X)$ and $|X|=|M-X|=|M|$ translate {\tt BLV}, which is inconsistent.  Thus {\tt BLV} and $|X|=|M-X|=|M|$ are jointly inconsistent, and by Proposition~\ref{p:interp-incons}, ${\tt BLV}(\eeq X)$ and $|X|=|M-X|=|M|$ are jointly inconsistent, which was to be demonstrated.  
\end{proof}

The sitation is no better for ${\tt BLV}(\eeq X)$ if $M$ is finite (in $\mathcal M$):
\begin{lem}\label{p:BLV-fin}
Suppose $\mathcal M \models 0< |X|<|M|<\omega$.  Then $\mathcal M\not\models {\tt BLV}(\eeq X)$.
\end{lem}
\begin{proof}
Let $X$ be given, and since $|M|>1$, we establish the following abbreviations: for $x \in X$ and $a\not\in X$, 
\begin{align}
X_x & = X-\{x\}\\
a+X_x &= X_x \cup \{a\}  
\end{align}
Notice that for all $x\in X$ and $a\not\in X$, $|a+X_x|\in [X]_\boxminus$ but $a+X_x \neq X$.  Thus if $\mathcal M \models {\tt BLV}(\eeq X)$, then $\varepsilon X \neq \varepsilon(a+X_x)$.  Moreover, for distinct $x,y\in X$, $a+X_x$ and $a+X_y$ are distinct (thus so are their abstracts), and for distinct $a,b\not\in X$, $a+X_x$ and $b+X_x$ are distinct (thus so are their abstracts).

Now $|\{a+X_x\mid x\in X\}|=|X|$ for each $a\not\in X$, and $|\{a+X_x\mid a\not \in X\}|=|M-X|$ for each $x\in X$.  Further for fixed $x_0\in X, a_0\not\in X$, 
\[ \{a_0+X_x\mid x\in X\}\cap  \{a+X_{x_0}\mid a\not \in X\} = \{a_0+X_{x_0}\} \]
Lastly note that by Permuation Invariance, $\varepsilon \emptyset \neq \varepsilon X, \varepsilon (a+X_x)$ for any $a\not\in X, x\in X$.  Thus, and since $\mathcal M\models |M|<\omega$, and letting
\begin{align}
A &= \{\varepsilon (a_0+X_x)\mid x\in X\} & B & = \{\varepsilon(a+X_{x_0})\mid a\not \in X\}
\end{align}
we have 
\begin{eqnarray}
 |M|& = &|M-X| + |X| \\
&=&   |A| + |(B - \{\varepsilon(a_0+X_{x_0})\}) \cup \{\varepsilon X\}| \\
&=& |A\cup B \cup \{\varepsilon X\}| \\
& < & |A\cup B \cup \{\varepsilon X\}| + |\{\varepsilon \emptyset\}| \\
& \leq & |\rng{\varepsilon}|\leq |M|
 \end{eqnarray}
which means that $|M|<|M|$, a contradiction.
\end{proof}
The preceding results on refinements of separations and on restrictions of {\tt BLV} yield:
\begin{cor}
\label{cor:sep-BLV}
If $E(\boxminus)_X$ refines a separation  and $\mathcal M \models |X|=|M-X|$, then  $\mathcal M \not \models A_E$.
\end{cor}
\begin{proof}
Corollary~\ref{cor:BLv-inf} and Lemma~\ref{p:BLV-fin} imply that (in the presence of \ref{ipm}) \begin{equation}
\mathcal M \models |X|=|M-X|\rightarrow \neg{\tt BLV}(\eeq X) \label{e:BLV}
\end{equation}  If $E(\boxminus)_X$ is a separation and $\mathcal M \models |X|<\omega$ then by Theorem~\ref{thm:main-formal} $\mathcal M \models {\tt BLV}(\eeq X)$ if $\mathcal M \models A_E$.  But this contradicts (\ref{e:BLV}). 

On the other hand, if $\mathcal M \models |X|\geq\omega$, then by Lemma~\ref{lem:refinterp}, $A_E$ translates ${\tt BLV}(\eeq X)$.  So $A_E$ cannot hold in $\mathcal M$ as again this contradicts (\ref{e:BLV}).
\end{proof}

\subsection{Refinements of complementations and ${\tt LCP}(\eeq X)$}\label{ssec:LCP}

The situation is quite similar for ${\tt LCP}$.  Here we handle the finite case first:  
\begin{lem}\label{p:complsmall}  
\begin{gather}
\label{p:complsmall-=X} {\tt LCP}(\eeq X) \models |X|=|M-X|<\omega \rightarrow |M|\leq 4 \\
\label{p:complsmall-unrest} {\tt LCP} \models |M|<\omega  \rightarrow |M|\leq 2
\end{gather}
\end{lem}
Here, abusing notation,
``$\models$'' indicates the derivability relation in second-order logic with well-behaved cardinalities.

We begin with the following:
\begin{prop}\label{p:more2}
	The following is a theorem of second-order logic:
	\begin{equation}
	(\forall X, Y)(|X|>2, |Y|\geq 2 \rightarrow |X\cup Y|\leq  |X\times Y|)
	\end{equation}
	with a strict inequality in the consequent if $X\cup Y$ is Dedekind finite.
\end{prop}
\begin{proof} 
	We assume without loss of generality that $X$ and $Y$ are disjoint, let $a, b, c\in X$ and $d, e\in Y$ be pairwise distinct.  Then define $f:X\cup Y \rightarrow X \times Y$ by 
	\[ f(x) = \begin{cases}
	(a, x) & x\in Y, x\neq d \\
	(x, d) & x\in X \\
	(b, e) & x = d 
	\end{cases} \]
\sloppy
	Clearly $f$ is injective;  it is not surjective since $(c, e)\not\in \rng{f}$.  Further, if $g:X\times  Y\rightarrow X\cup Y$ is a bijection, then $g\circ f$ witnesses that $X\cup Y$ is Dedekind infinite.
\end{proof}

\fussy

\begin{proof}[Proof of Lemma~\ref{p:complsmall}]
	For the first assertion, suppose the antecedent.  
	Reasoning deductively, suppose now that $|M|>4$, then $|X|, |M-X|>2$. 
	Set $X= A$ and $M-X = B$, then for each $(a, b)\in A\times B$, set $A_b = (A-\{a\})\cup \{b\}$ and $B_a = (B-\{b\}) \cup \{ a\}$.  Now for $(a, b), (c, d)\in A\times B$ distinct (as a pair), we have 
	\begin{align}
	\copyright A_b & = \copyright B_a & \copyright A_d & = \copyright B_c & 
	\copyright A_b & \neq \copyright A_c
	\end{align}
	So $f(x, y) = \copyright A_y= \copyright B_x$ defines an injection from $A\times B$ into $M = A \cup B$, contradicting Proposition~\ref{p:more2}.  Thus, $|M|\leq 4$.

	For the second assertion, suppose $|M|>2$.  Then $f(x) = \ell \{ x\}$ shows that 
	for $\rng{\ell}_1 = \{ \ell X \mid |X|=1\}$, $|\rng{\ell}_1| = |M|$ and so by finiteness $\rng{\ell}_1 = M$.  However for any $m\in M$, $\ell \{ m\} \neq \ell \emptyset$ by {\tt LCP} so $\ell \emptyset \not \in \rng{\ell}_1 = M$, a contradiction.
\end{proof}

Lemma~\ref{p:complsmall} would make it appear that ${\tt LCP}(\eeq X)$ must have mostly structures of infinite size.  But it has \emph{no} structures of infinite size:

\begin{prop}\label{p:comp-incons}
In the presence of \ref{ipm}, ${\tt LCP}(\eeq X)$ implies the universe is Dedekind finite.
\end{prop}
\begin{proof}
Given \ref{ipm}, if $|M|\geq\omega$, there is a concept $X$ such that $|X|=|M-X|=|M|$.  We show that ${\tt LCP}(\eeq X) \wedge |M|\geq \omega$ translates ${\tt BLV}(\eeq X)$, so then by Corollary~\ref{cor:BLv-inf} and Proposition~\ref{p:interp-incons}, the former is inconsistent. 

Let $f:M\rightarrow X$ and $g:M\rightarrow M-X$ be bijections, and define $\varepsilon$ by 
\[ \varepsilon Y =  \begin{cases}
f(\ell Y) & \ell \emptyset \in Y \\
g(\ell Y) & \ell \emptyset \not\in Y
\end{cases} \]

Clearly $\varepsilon$ is functional at $X$.  Now if $Y,Z\in [X]_\boxminus$,  $\ell \emptyset \in Y, Z$ and $Y\neq Z$, then by ${\tt LCP}(\eeq X)$ and the injectivity of $f$, $\varepsilon Y \neq \varepsilon Z$.  If $\ell \emptyset \in Y-Z$, then as $f(M)\cap g(M)=\emptyset$, $\varepsilon Y \neq \varepsilon Z$.  Lastly if $\ell\emptyset \not\in Y, Z$, then by ${\tt LCP}(\eeq X)$ and the injectivity of $g$, $\varepsilon Y \neq \varepsilon Z$.  
\end{proof}

Together with Lemmata~\ref{lem:refinterp} and~\ref{p:complsmall}, from Proposition~\ref{p:comp-incons} we obtain 
\begin{cor}\label{cor:comp-CP} 
If $\mathcal M\models |X|=|M-X|$ and $E(\boxminus)_X$ refines a complementation, then 
$\mathcal M \models A_E  \rightarrow |M|\leq 4$.
\end{cor}
\begin{proof}
Propositions~\ref{p:complsmall} and~\ref{p:comp-incons} together give that 
\begin{equation}
\label{e:comp-implic}
\mathcal M\models |X|=|M-X|\geq 4 \rightarrow \neg {\tt LCP}(\eeq X)
\end{equation}
If $\mathcal M \models |X|<\omega$ and $E(\boxminus)_X$ refines a complementation, then by Theorem~\ref{thm:main-formal} $\mathcal M \models {\tt LCP}(\eeq X)$ if $\mathcal M\models A_E$.  But then (\ref{e:comp-implic}) gives that $|X|=|M-X|<4$.

On the contrary if $\mathcal M\models |X|\geq \omega$, then as $E(\boxminus)_X$ refines a complementation,
\end{proof}

\subsection{Applying the bicardinal classification}

Having dealt individually with each kind of non-trivial equivalence relation, we now apply Theorem~\ref{thm:main-formal}, for putting together Corollaries~\ref{cor:sep-BLV} and~\ref{cor:comp-CP} with Theorem~\ref{thm:main-formal} gives the following:
\begin{thm}\label{thm:BC}
If $\mathcal M \models 2<|X|=|M-X|$ and $E(\boxminus)_X$ is nontrivial, then $\mathcal M \not\models A_E$.\end{thm}
\begin{proof}
By Theorem~\ref{thm:main-formal}, if $E(\boxminus)_X$ is not trivial it either refines a separation or refines a complementation.  In the former case $A_E$ is inconsistent by Corollary~\ref{cor:sep-BLV}, and in the latter case $A_E$ is inconsistent by Corollary~\ref{cor:comp-CP} and by our assumption that $|M|>4$.
\end{proof}
\begin{rmk}
The converse of Theorem~\ref{thm:BC} fails, for let $A_E$ be the abstraction principle {\tt NewV}:
\[ (\forall X, Y)(\varsigma X = \varsigma Y \leftrightarrow (|X|=|Y|=|M| \vee (|X|,|Y|<|M|\nd X=Y)) \] 
But by K\"onig's theorem, {\tt NewV} is has no standard models whose first-order domain is a singular limit cardinal.\footnote{There is in fact a small error on just this point at~\cite[p.~596]{WalshED2015} in stating the results of~\cite[p.~315]{Shapiro1999ab}:  ``For instance, the claim that the abstraction principle New V is strongly stable is equivalent to the generalized continuum hypothesis.''  The correct statement is this:  {\tt New V} is not strongly stable, and whether it is stable depends on whether $2^\kappa = \kappa^+$ for unboundedly many $\kappa$.  
} 
\end{rmk}

To put Theorem~\ref{thm:BC} in the terms used earlier:  If $M$ is large enough (having greater than four elements), then $\mathcal M$ satisfies no abstraction principle whose equivalence relation is non-trivial on concepts which evenly divide the universe.
So Theorem~\ref{thm:main-formal} gives a mark by which to identify (at least) many bad companions:  these insufferable principles limit the size of the universe by non-trivially carving concepts of maximally large bicardinality.\footnote{Theorem~\ref{thm:BC} is in some ways a companion to Fine's \emph{Characterization} Theorem, mentioned earlier:  that what he calls the \emph{basal} abstraction principle is the finest satisfiable in all infinite domains.  As Theorem~\ref{thm:main-formal} is a deductive analogue of Fine's Classification Theorem, a deductive analogue of his Characterization Theorem can also be obtained from Theorem~\ref{thm:main-formal}, though we leave this for Appendix~\ref{sec:reFine}.}

\section{Classification and relative categoricity}\label{sec:relcat}

In~\cite{WalshED2015} Walsh and Ebels-Duggan introduce a criterion that distinguishes the abstraction principle {\tt HP}.  The so-called ``Julius Caesar problem''\footnote{See \cite[\S~55]{Frege1980}, and for example~\cite{Heck1997} and~\cite[Chapter~14]{Hale2001aa}.} arises from the following fact.  There are abstraction principles, among them {\tt HP}, with the following unfortunate property: for a given base $L_0$-structure, it is possible to interpret $\partial$ in two ways, both satisfying the abstraction principle, but such that the two expanded structures are neither isomorphic nor elementarily equivalent.  That this can be done with the equivalence relation $\approx$ and {\tt HP} was noted by Frege himself~\cite[\S~56, \S\S 66ff]{Frege1980}.  Walsh and Ebels-Duggan showed that though these abstraction principles lack \emph{these} properties, some of them have  weaker but still noteworty equivalence properties.

Abusing notation somewhat, let $A_E[\partial]$ be not only the abstraction principle correlated with $E$, but the theory containing as axioms that abstraction principle and the axioms of our background logic (including those for well-behaved cardinalities).  With similar abuse, we'll use $A_E^2[\partial_1, \partial_2]$ be a theory containing axioms for our background logic  and the abstraction principles $A_E[\partial_1]$ and $A_E[\partial_2]$.  In other words, $A_E^2[\partial_1, \partial_2]$ is a theory with two copies of the abstraction principle $A_E$, one for each abstraction operator.  To introduce Walsh and Ebels-Duggan's weakened equivalence notions we first introduce the relevant notion of an isomorphism between induced models:
\begin{defn}\label{def:induced}
Given a structure $\mathcal M = ( M, S_1[M], S_2[M], \ldots, \partial_1, \partial_2)$ satisfying $A_E^2$, for $i=1,2$, let 
\begin{equation}\label{e:induced}
\mathcal M_i = ( \rng{\partial_i}, S_1[M] \cap P(\rng{\partial_i}), S_2[M] \cap P(\rng{\partial_i}\times \rng{\partial_i}) \ldots, \partial_i) 
\end{equation}
We say that $\mathcal M_i$ is the model \emph{induced} by $\partial_i$.  

As noted in~\cite{WalshED2015}, for $\Gamma$ to be an isomorphism between $\mathcal M_1$ and $\mathcal M_2$ it is sufficient for $\Gamma$ to satisfy the following condition:  For all $X\in S_1[M]\cap P(\rng{\partial_1})$, 
\begin{equation}
\label{e:isomcond}
\Gamma \partial_1 X = \partial_2 (\Gamma X) 
\end{equation}
\end{defn}

\begin{defn}\label{dfn:RCetc}  \sloppy
The theory $A_E$ is \emph{naturally relatively categorical} just if for any model $\mathcal M$, and any $\partial_1,\partial_2$ such that $\mathcal M\models A_E^2[\partial_1,\partial_2]$, the \emph{natural bijection} $\Gamma: \rng{\partial_1} \rightarrow \rng{\partial_2}$, defined by 
\begin{equation}
\label{e:natbij} \Gamma \partial_1 X = \partial_2 X 
\end{equation}
is an isomorphism between the induced models $\mathcal M_1$ and $\mathcal M_2$.  

\fussy
The theory $A_E$ is \emph{relatively categorical} just if for any model $\mathcal M$, and any $\partial_1,\partial_2$ such that $\mathcal M\models A_E^2[\partial_1,\partial_2]$, $\mathcal M_1 \cong \mathcal M_2$.

The theory $A_E$ is \emph{relatively elementarily equivalent} just if for any model $\mathcal M$, and any $\partial_1,\partial_2$ such that $\mathcal M\models A_E^2[\partial_1,\partial_2]$, 
\[ \mathcal M_1 \models \varphi_{\partial_1} \Leftrightarrow \mathcal M_2 \models \varphi_{\partial_2} \] 
for every \emph{sentence} (i.e., closed formula) $\varphi$ in the language of the theory $A_E[\partial]$, and 
$\varphi_{\partial_i}$ is the result of replacing every occurrence of $\partial$ in $\varphi$ with $\partial_i$.
\end{defn}
Thus natural relative categoricity is distinguished from relative categoricity in that the former requires that the natural bijection be an isomorphism, the latter only requires that an isomorphism exist.

One of the main theorems of~\cite{WalshED2015} established necessary and sufficient conditions for the \emph{natural} relative categoricity of an abstraction principle;  some of which we state as:
\begin{NRCThm}[\cite{WalshED2015}]\label{thm:NRCCC}
Let $A_E$ be an abstraction principle.  The following are equivalent:
\begin{enumerate}
\item $A_E$ is naturally relatively categorical.
\item \label{it:ccoa} $A_E[\partial] \models (\forall X,Y)(|Y|=|X|\leq |\rng{\partial}|\rightarrow E(X,Y))$.
\end{enumerate}			
\end{NRCThm}
\noindent where again ``$\models$'' means the deductive consequence relation for our strong background logic.  
Abstraction principles satisfying this second condition are said to be \emph{cardinality coarsening on abstracts}.

As can be seen plainly, $\approx$ is the finest equivalence relation such that $A_E$ is cardinality coarsening on abstracts;  thus {\tt HP} is the finest naturally relatively categorical abstraction principle.  This gives the neo-logicist a criterion by which to distinguish {\tt HP} from other abstraction principles;  notably, the abstraction principle  {\tt NewV}, which is equivalent to ${\tt BLV}(<\!\!|M|)$, is \emph{not} naturally relatively categorical (see~\cite[\S~5]{WalshED2015}).  

But there is more than can be asked.  In~\cite[Proposition~14, p.~1687]{Walsh2012aa} Walsh shows  that {\tt HP} is relatively categorical in the stronger, unqualified sense.  Moreover the foregoing establishes the following implication relations:
\begin{equation}\label{d:RCarrows} \begin{array}{ccc}
\textrm{Natural relative categoricity} & \Leftrightarrow & \textrm{Cardinality coarsening on abstracts}\\
\Downarrow \\
\textrm{Relative categoricity} \\
\Downarrow \\
\textrm{Relative elementary equivalence} 
\end{array}\end{equation}

Walsh and Ebels-Duggan thus raised two questions in~\cite{WalshED2015}:    
\begin{Q}\label{q:RC}
Are the conditions for relative categoricity the same as for natural relative categoricity?  Can we drop the specification of the natural bijection in the~\ref{thm:NRCCC}?  More precisely is (\ref{it:ccoa}) equivalent to the relative categoricity of $A_E$?
\end{Q}

\begin{Q}\label{q:REE}
What is the relation between relative categoricity, natural relative categoricity, and relative elementary equivalence?
\end{Q}

It is a consequence of Theorem~\ref{thm:main-formal} that Question~\ref{q:RC} can be answered in the affirmative, and that, in answer to Question~\ref{q:REE}, all the implication arrows of (\ref{d:RCarrows}) can be reversed.

From here onward it will be helpful to recall the definitions of restricted abstraction principles from page~\pageref{p:restap}.

\begin{lem}\label{lem:BLVnotRC}
	Suppose $A_E$ is an abstraction principle, and that $A_E[\partial]$ translates ${\tt BLV}(\leeq{} 1)$ via $\varphi(Y,Z)$. 
	Then $A_E[\partial]$ is not relatively elementarily equivalent, and so not relatively categorical.
\end{lem}
\begin{proof}

	
	Let $\partial_1$ be given such that $\mathcal M\models A_E[\partial_1]$.  
	Note first that as $A_E$ translates ${\tt BLV}(\leq 1)$, $|\rng{\partial}|=|M|\geq \omega$.  We may thus assume that $\rng{\partial_1}= M$.  The $\partial_2$ we will construct will have the same range; this allow us to avoid worries that the translation $\varphi$ of ${\tt BLV}(\leq 1)$ in $\mathcal M$, which may contain nested quantifiers, may change its behavior on the induced models $\mathcal M_1$ and $\mathcal M_2$.
	
	For $\varphi$ translating as above, let $x\in U(X)$ if and only if $\varphi(X, Y)$ and $x\in Y$.  Consider now \begin{equation}
	\psi_i = (\exists a)(a = \partial_i U(\{a\})) \label{e:disagreesent}
	\end{equation}
	
	On the one hand, if $\mathcal M \models \neg \psi_1$, then select $a\in M$, and note that there are $b, c\in M$ with $b  = \partial_1 U(\{a\})$ and $ a  = \partial_1 U(\{c\})$.  Let $f(b)=a$, $f(a)=b$, and $f$ the identity map on $M-\{a,b\}$.  As $f\in \mathcal M$ since it is definable, setting $\partial_2 X = f(\partial_1X)$ we have that $\mathcal M \models A_E^2[\partial_1, \partial_2]$.  And yet $\mathcal M_2\models \psi_2$.
	
	On the other hand, if $\mathcal M\models \psi_1$, let 
	\begin{equation}\label{e:defAwithoutdivby2}
	A = \{ \partial_1 X\mid (\exists a)(a= \partial_1 X = \partial_1 U(\{a\}))\}
	\end{equation}
	
	Suppose first that $|A|=|\rng{\partial_1}|$.  As $|A|\geq\omega$, there is a $B\subset A$ such that $|B|=|A-B| = |A|$ by \ref{ipm};  let $g: B \rightarrow A-B$ be a bijection.  Then let 
	\begin{equation}\label{e:g-def-nodivby2}
	\partial_2 X = \begin{cases}
	g(\partial_1 X) & \partial_1 X \in B \\
	g^{-1}(\partial_1 X) & \partial_1 X \in g(B) \\
	\partial_1 X & \partial_1 X \not\in A
	\end{cases}
	\end{equation}
	Again we have that $\mathcal M \models A_E^2[\partial_1, \partial_2]$, and $\mathcal M_2\models \neg \psi_2$: for if $\partial_1 U(\{a\}) = a \in B$ then $\partial_2 U(\{a\}) \in A-B$, and so $\partial_2U(\{a\}) \neq a$.  And likewise if $\partial_1U(\{a\}) = a \in A-B$.  
	
	Lastly suppose $|A|<|\rng{\partial_1}|$, then as $|\rng{\partial_1}|\geq \omega$, then by \ref{ism} there is an injection $g:A \rightarrow \rng{\partial_1} - A$.  Setting $B=A$,  the construction as in \eqref{e:g-def-nodivby2} again delivers the verdict that $\mathcal M_2\models \neg \psi_2$.
\end{proof}

Lastly, we prove a parallel to Lemma~\ref{lem:refinterp}.

\begin{lem}
	\label{lem:interp1}
	Suppose $\mathcal M \models 0<|X|<\omega\leq |M|$ and $E(\boxminus)_X$ refines a separation.  Then $A_E$ translates ${\tt BLV}(\leeq 1)$.
\end{lem}
\begin{proof}
 By Proposition~\ref{p:pairingimage}, for any singleton $Y$, $\langle Y, X\rangle \in [X]_\boxminus$.  Further, if $Y\neq Z$ and both are singletons, then $\langle Y,X\rangle$ and $\langle Z, X\rangle$ are disjoint, so $|\langle Y,X\rangle \triangle \langle Z, X\rangle|\not < |X|$.  Thus, setting $\varepsilon Y = \partial \langle Y, X\rangle$, we have for singletons $Y$ and $Z$:
	\[ \varepsilon Y = \varepsilon Z \Leftrightarrow E(\langle Y, X\rangle, \langle Z, X\rangle ) \Leftrightarrow \langle Y, X\rangle = \langle Z, X\rangle \Leftrightarrow Y=Z \] 
\end{proof}

We now answer Questions~\ref{q:RC} and~\ref{q:REE}:
\begin{thm}\label{thm:solveWED}
An abstraction principle $A_E$ is relatively categorical if and only if it is cardinality coarsening on abstracts.

Further, an abstraction principle is relatively elementarily equivalent if and only if it is   relatively categorical.
\end{thm}
\begin{proof}
The right-to-left direction of the second assertion is trivial; for the right-to-left direction of the first, suppose $A_E$ is cardinality coarsening on abstracts, and let $\mathcal M \models A_E^2[\partial_1, \partial_2]$.  Let $\Gamma$ be the natural bijection defined by $\Gamma\partial_1 X = \partial_2 X$;  note that as $|\Gamma X| = |X|$, since $A_E$ is cardinality coarsening on abstracts, in 
 $A_E^2[\partial_1,\partial_2]$ it follows that $\partial_i X = \partial_i \Gamma X$.  Thus   
\begin{equation}
\Gamma \partial_1 X = \partial_2 X = \partial_2 \Gamma X
\end{equation}
so $\Gamma$ is an isomorphism.\footnote{This direction of the proof generalizes the proof in~\cite[Proposition~14]{Walsh2012aa}.}

For the left-to-right direction of both assertions, observe that the following exhaust the possibilities for the cardinal placement of any two equinumerous concepts:
\begin{align}
\mathcal M \models &|M|=|Y|=|X|=|M-X| = |M-Y| 	\label{e:XYTop} \\
\mathcal M \models &|M|=|Y|=|X|=|M-X| >|M-Y| 	\label{e:XTop>Y} \\
\mathcal M \models &|M|=|Y|=|X|>|M-X|\geq |M-Y| 
\label{e:Top>X>Y}\\
\mathcal M \models & \omega\leq |X|=|Y| < |M|		\label{e:XYinfnotTop} \\
\mathcal M \models &|X|=|Y|<\omega \leq |M|		\label{e:XYfinMinf} \\
\mathcal M \models &|X|=|Y|\leq|M|<\omega			\label{e:XYMfin}
\end{align}
To prove the second assertion, then it suffices to show that on each of these possibilities, if 
$\mathcal M\models A_E^2[\partial_1, \partial_2] $ then either (i) $\mathcal M_1$ and $ \mathcal M_2$ are isomorphic (and so elementarily equivalent), or (ii) there are $\partial_1, \partial_2$ such that $\mathcal M_1$ and $\mathcal M_2$ are not elementarily equivalent (and so not isomorphic).   From the right-to-left direction of the first assertion, we  know that if $E$ is cardinality coarsening on abstracts, then (i) holds for all possibilities. Thus to prove the second assertion, it suffices to prove that if $E$ is not cardinality coarsening on abstracts, then either $\mathcal M \not\models A_E$, or there are $\partial_1, \partial_2$ with $\mathcal M \models A_E^2[\partial_1, \partial_2]$ but $\mathcal M_1$ and $\mathcal M_2$ are not elementarily equivalent.  This will prove the first assertion as well.  

So suppose that $A_E$ is not cardinality coarsening on abstracts; let $\mathcal M \models A_E$ witness this failure.  So there are  $X, Y\in \mathcal M$ such that 
\begin{equation}
\label{e:iso-witness} \mathcal M \models |Y|=|X|\leq |\rng{\partial_1}| \wedge \neg E(X,Y)
\end{equation} Clearly then $\mathcal M\models 0<|X|$. 

The first case in which (\ref{e:XYTop}) holds can be ruled out since if $E(\boxminus)_X$ is non-trivial, then by Corollaries~\ref{cor:sep-BLV} and~\ref{cor:comp-CP}, $\mathcal M \not\models A_E$.  Thus for the second case (\ref{e:XTop>Y}) we must assume that $E(\boxminus)_X$ is trivial.  Without loss of generality assume $\mathcal M \models X\subset Y$.  Since by assumption (\ref{e:iso-witness}), $\mathcal M \models |M|=|X|=|Y|\leq |\rng{\partial}|$,
let $f, g \in \mathcal M$ be respective witnesses of $|X|=|\rng{\partial}|$ and $|Y|=|\rng{\partial}|$, then set 
$\partial_1 = f \circ \partial$ and $\partial_2 = g \circ \partial$. 
Since $\partial_1, \partial_2$ are both definable injections in $\mathcal M$, it follows that $\mathcal M \models A_E^2[\partial_1, \partial_2]$.  Note, however, that for all $Z\subset X$, if $|Z|=|X|$ then $Z\boxminus X$.  Hence, since $E(\boxminus)_X$ is trivial, $E(X,Z)$.  Thus, as 
\[ \mathcal M_1 \models (\forall x)(Xx \leftrightarrow x=x) \] 
it follows that 
\begin{equation}
\label{e:BCnotREE}
\mathcal M_1 \models (\forall U,W)((\forall x)(Ux \leftrightarrow x=x) \wedge |W|=|U|\rightarrow \partial_1 U = \partial_1 W)
\end{equation}
However, notice that 
\[ \mathcal M_2 \models (\forall x)(Yx \leftrightarrow x=x) \] 
and that $f\circ g \in S_2[M]$ and $f\circ g\subset Y\times Y$.  Thus $f\circ g\in \mathcal M_2$, so 
\[ \mathcal M_2 \models |X|=|Y| \wedge \partial_2 X \neq \partial_2 Y \]
as by assumption $X\subset Y$ and 
$\mathcal M \models A_E[\partial_2]\wedge \neg E(X,Y)$.
Consequently, we see that 
\[\mathcal M_2 \not\models (\forall U,W)((\forall x)(Ux \leftrightarrow x=x) \wedge |W|=|U|\rightarrow \partial_2 U = \partial_2 W) \]
So by (\ref{e:BCnotREE}) $A_E$ is not relatively elementarily equivalent, and so not relatively categorical.  

For (\ref{e:Top>X>Y}) observe that  by \ref{ipm} if $|M|=|X|\geq \omega$ then there is $Z\subset X$ such that $|Z|=|M-Z|=|M|$; by (\ref{e:XYTop}) $E(\boxminus)_Z$ is trivial, and so  by (\ref{e:XTop>Y}) both $X$ and $Y$ are in $[Z]_E$.

If (\ref{e:XYinfnotTop}) holds, then by Theorem~\ref{thm:main-formal}, $E(\boxminus)_X$ is non-trivial, but then by Lemma~\ref{lem:refinterp} 
$A_E$ translates ${\tt BLV}(\leeq X)$ on $\mathcal M$.  
Likewise if (\ref{e:XYfinMinf}), then by Lemma~\ref{lem:interp1} we see again that $A_E$ translates ${\tt BLV}(\leq 1)$ on $\mathcal M$. 
But then by Lemma~\ref{lem:BLVnotRC} $A_E$ is not relatively elementarily equivalent, and so not relatively categorical.  

Finally if (\ref{e:XYMfin}) holds then by Theorem~\ref{thm:main-formal}, as $0<|X|<|M|$, by Theorem~\ref{thm:BC} and $\mathcal M \models A_E$, we have that $|M|\leq 4$ and $\mathcal M \models {\tt LCP}(\eeq X)$.  If $|M|=2$ then if $E(\boxminus)_X$ refines a complementation then it is trivial (and so cardinality coarsening on abstracts),
so $|M|=4$.  But then by the argument of~\cite[\S~5.5, p.~594]{WalshED2015}, $A_E$ is not relatively elementarily equivalent, and so not relatively categorical.
\end{proof}

Walsh and Ebels-Duggan prove another theorem~\cite[Theorem~1.2]{WalshED2015} for another weakened equivalence property:  that if all objects are abstracts, then the induced models for $A_E$ are isomorphic via the natural bijection if and only if bicardinally equivalent concepts are $E$-equivalent.  They ask an analogous question as well:  can one drop the assumption that the natural bijection is an isomorphism and still obtain the biconditional?  The proof of Theorem~\ref{thm:solveWED} indicates an affirmative answer here also:  to adapt the proof one must observe that (\ref{e:XTop>Y}) is irrelevant on the assumption that $X$ and $Y$ are bicardinally, but not $E$-, equivalent.  The other cases remain the same.  

Thus we have:
\begin{thm}\label{thm:solveWED2}
    Suppose that $\mathcal M \models A_E^2[\partial_1, \partial_2] \nd \rng{\partial_1}=\rng{\partial_2} = M$.  Then $\mathcal M_1 \cong \mathcal M_2$ if and only if $A_E$ is bicardinality coarsening on abstracts.  
\end{thm}

It may seem surprising to see in Theorem~\ref{thm:solveWED} the equivalence of relative elementary equivalence with relative categoricity.  But reflection suggests that this shouldn't be so striking after all:  for Theorem~\ref{thm:solveWED} depends  on Theorem~\ref{thm:main-formal}, and so on the equivalence relations in play being permutation invariant because $L_0$-definable.  The equivalence is in a sense an artifact of this restriction.  It is an open question as to whether 
other invariance conditions (see~\cite{Fine2002}, \cite{Antonelli2010aa}), \cite{Cook2016doi})
 would yield the same results.  

However, reflection on this equivalence could be taken as philosophically relevant, and indeed good news for the neo-logicist.  For it is trivial, by \ref{p:perminv}, that in an $L_0$-structure $\mathcal M$, if $\pi$ is a permutation in $\mathcal M$, then the structure resulting from that permutation is both isomorphic and elementarily equivalent to $\mathcal M$.  This is trivial, of course, because the permuted structure just is the original structure.  But it is noteworthy that this triviality hides the coinstantiation of isomorphism and elementary equivalence.  

In light of this, one might ask:  is this coinstantiation preserved under abstraction?  In more detail:  given two abstraction operators, $\partial_1, \partial_2$ for an abstraction principle $A_E$, the natural bijection $\Gamma(\partial_1 X) = \partial_2$ will always be in any model of $A_E[\partial_1,\partial_2]$.  Therefore if $\partial_1$ and $\partial_2$ share the same range, then $\mathcal M_2$ can be regarded as the structure resulting from the permutation $\Gamma$ of the structure $\mathcal M_1$.  Is it the case that under any such permutation, the resulting models will be isomorphic if and only if elementarily equivalent?  Indeed, answers Theorem~\ref{thm:solveWED}, the answer is affirmative:  this property of coinstantiation is,  in the given sense, preserved under permuations of abstractions.  One might  argue on behalf of the neo-logicist (though we will not) that this shows abstraction is in some sense a ``logical'' operation.  

\section{Conclusion}

We began with the promise that strengthening Fine's Classification Theorem would advance our understanding of the demarcation of logical abstraction principles.  The strengthened version of Fine's theorem yields a strengthened relative categoricity theorem; the relevant advancement thus comes in the option provided by the latter result.  

Walsh and Ebels-Duggan note that relative categoricity as we have described it cannot rule out bad companions like {\tt NP}:  since {\tt NP} has only Dedekind finite models, it (the \emph{principle}) is cardinality coarsening on abstracts.  They also note, however, that the equivalence relation underlying {\tt NP} is not cardinality coarsening \emph{on small concepts}---a notion that applies not to abstraction principles but to equivalence relations.  (See \cite[section~5.4]{WalshED2015}.)

The confluence of results presented in this paper offers a stronger statement of this point.  As we can see from the results in sections~\ref{sec:badco} and~\ref{sec:relcat}, when nontrivial equivalence relations are manifest in a structure, they generate failures of relative categoricity.  It \emph{is} curious that {\tt NP} is relatively categorical as a principle.  But it is more relevant that the equivalence relation deployed by {\tt NP} is nontrivial on all infinite bicardinal slices---it has non-trivial manifestations.  And wherever nontriviality appears in a structure satisfying an associated abstraction principle, it generates a failure of relative categoricity.  Thus, even though {\tt NP} is by the letter relatively categorical, it deploys an equivalence relation that generates failures of relative categoricity.  On this line of thinking, {\tt NP} can thus be ruled out.  The same goes for other bad companions.

This is a sketchy argument for several reasons, but for our purposes it needn't be more.  As we said at the outset, our goal is to provide and clarify options.  All things considered, this sort of account may be rejected.  But the results here adduced, along with those in~\cite{WalshED2015} and~\cite{Walsh2012aa}, indicate that the sort of answer just offered is one of the things that should be considered.  

In light of this, it may be worth considering relative categoricity as kin to permutation invariance.  Both are motivated by the idea that logic is indifferent to the particulars of objects.  Permutations represent this indifference by exchanging objects arbitrarily, while relative categoricity represents it by treating of arbitrarily selected abstraction assignments.  Here, more issues loom.  Our discussion has been organized around permutation invariance, but there are stronger conditions claiming to be necessary for logicality.  Permutation invariance is \emph{intra-modular}:  permuations simply re-order or re-organize the elements of a given structure.  Recent work has focused on the \emph{trans-modular} notion of \emph{isomorphism between structures}.  This is arguably a better mark of the logical than permutation invariance, and better matches the intuitive notion of indifference to objects.  Though the language $L_0$ is logical in this stronger sense, we have not addressed how a notion like relative categoricity would apply in a trans-modular setting.\footnote{Jack Woods, in~\cite{JWoods2014}, has addressed the question of \emph{indefinite} abstraction principles in a trans-modular setting.  Given a domain $D$ and an equivalence relation $E$ on $P(D)$, let 
    \[ F(D)_E =  \{ f:P(D)\rightarrow D \mid (\forall X,Y\subset D)(fX=fY \Leftrightarrow E(X,Y))\} \]
    In this setting of \cite{JWoods2014}, the members of $F(D)_E$ are called abstraction \emph{functions} for $E$, while an abstraction \emph{operator} for $E$ is a function $\sigma$ taking domains $D$ as an argument, outputting a set $\sigma^D\subseteq F(D)_E$.  If $\zeta:D\rightarrow D'$ is a bijection, then $\zeta$ can be extended to an isomorphism $\zeta^+$ of all types over $D$ and $D'$.  As such, $\zeta^+(\sigma^D)$ will be $\{\zeta^+ \circ f\circ (\zeta^+)^{-1} \mid f\in \sigma^D\}$ (see \cite[pp.~281ff]{JWoods2014}).  The abstraction operator $\sigma$ is then \emph{isomorphism invariant} just if $\zeta^+(\sigma^D) = \sigma^{D'}$ for all $D, D'$, and bijections $\zeta:D\rightarrow D'$.  An abstraction operator is said to be \emph{full} just if for any $D$, $\sigma^D =F(D)_E$.  Woods endorses the view that indifference to particular objects is the mark of the logical, and thus that isomorphism invariance is the mark of the logical.  
    
    Woods offers as evidence for this claim the proposition on \cite[p.~298]{JWoods2014} that all and only abstraction operators logical in the sense of isomorphism invariance are both full and associated with an isomorphism invariant (in a broader sense) equivalence relation.  However, this proposition, and the lemma used to prove it (presented on \cite[p.~296]{JWoods2014}), are incorrect.  The lemma in question claims that if $\sigma$ is (non-empty and) invariant then it is full; but the proof makes two incorrect assumptions:  the first is that if $f, g \in F(D)_E$ and $\rng{f} \approx \rng{g}$ then $M-\rng{f} \approx M- \rng{g}$ \cite[p.~296]{JWoods2014}.  The second is that if $\zeta(f(A))=g(A)$ for $f, g\in F(D)_E$ and all $A\subseteq D$, then $\zeta^+(f)=g$.  But this is not in general true, for the extension of $\zeta$ to $\zeta^+$ has $\zeta^+(f) = \zeta^+ \circ f \circ (\zeta^+)^{-1}$, and while this implies that $\zeta^+(f)(A) = g((\zeta^+)^{-1}(A))$, the last $g((\zeta^+)^{-1}(A)) = g(A)$ holds only, in general, when $E$ is coarser than $\approx$;  see \cite[Theorems~1.1 and~1.2]{WalshED2015} and Theorem~\ref{thm:solveWED}.  
    
    The lemma and resulting proposition can be shown false by the example of $\approx$ for $E$ and $\sigma^D = \{f\in F(D)_E\mid \rng{f}= D\}$.  Clearly this is not full, and it can be shown isomorphism invariant using the techniques of the lemma on \cite[p.~297]{JWoods2014}.  This would suggest suitable restrictions on the cardinality or identity of the range of the abstraction functions might repair the lemma and proposition.  But more robust counterexamples can be generated by the results of the present paper.  For these we will first need the following observation.
    
        Suppose $\mathcal M[\partial_1, \partial_2]$ with first-order domain $M$ is standard and witnesses the fact that $A_E$ is not relatively categorical.  Then if $\zeta^+:\mathcal M\rightarrow \mathcal M'$ is an isomorphism, then $\mathcal M'[\zeta^(\partial_1), \zeta^(\partial_2)]$ also witnesses that $A_E$ is not relatively categorical, for $\mathcal M_i \cong \mathcal M'_i$ for $i=1,2$.  

The main idea of the counterexample is to choose a pair $\partial_1, \partial_2$ of surjective abstraction functions such that $\mathcal M_1 = \mathcal M[\partial_1] \not\cong \mathcal M_2 = \mathcal M[\partial_2]$.  Then for $M$ the domain of $\mathcal M$, we let $\sigma^M$ be the set of all abstraction functions $\partial$ such that $\mathcal M[\partial]$ \emph{is} isomorphic to $\mathcal M[\partial_1]$.  Clearly $\partial_2$ will not be among these.  Then, for every domain $D$, we set $\sigma^D$ to be the set containing $\zeta^+(\partial)$ for each $\partial \in \sigma^M$ and each   bijection $\zeta:M\rightarrow D$.  This ensures $\sigma$ will be isomorphism invariant, but also ensures, by the observation given above, that it will not be full.
    
    More formally:  it is a consequence of Theorem~\ref{thm:solveWED2} that there are abstraction principles with surjective abstraction functions $\partial_1, \partial_2$ on a domain $M$ such that for no bijection $\zeta:M\rightarrow D$ is there a bijection $\pi:D\rightarrow D$ such that 
    \[ \pi (\zeta^+ (\partial_1)(A)) = \zeta^+ (\partial_2 )(\pi(A)) \]
    (That is, the bijection $\pi$ does not commute with $\zeta^+(\partial_1)$ and $\zeta^+(\partial_2)$.)  This is true of {\tt NewV}, as demonstrated by Theorem~\ref{thm:solveWED2} in conjunction with the results of \cite[Section~5.2]{WalshED2015}.
    
    Thus for each $D\approx M$, set 
    \[ \sigma^{D} = \{\partial \in F(D')_E \mid \partial = \zeta^+ \circ \partial_1 \circ (\zeta^+)^{-1} \; \textrm{for some bijection $\zeta:M \rightarrow D$}  \}  \]
    The indefinite abstraction operator $\sigma$ is then isomorphism invariant by construction, but $\zeta^+(\partial_2) \not\in \sigma^{\zeta(D)}$ for any bijection $\zeta$, so $\sigma$ is not full.  Applied to the cited case of {\tt NewV} developed in \cite{WalshED2015}, we might choose $\partial_1$ such that the natural membership relation derived from it is well-founded.  Thus for any $D$ and any $f\in \sigma^D$, $f$ will foster such well-founded relations as well.  But for each $D$ there will always be $\partial_2 \in F(D)_E$ such that the natural membership relation is not well-founded, and thus such $\partial_2$ will be omitted from each $\sigma^D$.
    
    Though the lemma and proposition are false, the above counterexamples do not obviously determine the correctness of Woods's overall assertion:  that isomorphism invariance (for abstraction operators) is the mark of logic's characteristic indifference to the particularities of objects.  The second counterexample \emph{does} depend on the choice of an abstraction function, but \emph{not} on the choice of any first-order object.  This accords with Woods's prediction.}

One last point is worth making.  The results of section~\ref{sec:badco} can be recast.  Say that \emph{$A_E$ proves the universe set-like}---in symbols, $A_E\models {\tt ZFC}(M)$---just if \ref{cc}, \ref{ism}, \ref{ipm}, and every instance of \ref{cwf} is provable in the theory $A_E$.  The results of section~\ref{sec:badco}, put into this context, say that if $A_E$ proves the universe set-like, and if $E$ is provably non-trivial at ``large enough'' bicardinalities, then $A_E$ is inconsistent.  For if $A_E \models {\tt ZFC}(M)$ and $E$ is non-trivial at a large enough
 cardinality, then since {\tt ZFC}$(M)$, at that cardinal $E$ is either separative or complementative.  Then following the theorems of that section, \ref{ipm} renders the principle inconsistent.  This can be summed up by:
\begin{cor}\label{cor:nec}
If a theory $A_E$ proves $|M|>4$ then $A_E$ is inconsistent (without the help of our cardinality principles) if and only if $A_E$ proves that there is a well-ordering of $M$ and $[X]_{E(\boxminus)}$ is non-trivial for some $X$ with $|X|=|M-X|=|M|$.
\end{cor}
The non-trivial direction follows because our cardinality principles follow from a well-ordering of the universe.  

Similarly we obtain necessary and sufficient conditions for satisfiability (possession of a standard model):
\begin{cor}\label{cor:satnec}
If a theory $A_E$ proves $|M|>4$ then $A_E$ is unsatisfiable if and only if $[X]_E(\boxminus)$ is non-trivial for some $X$ with $|X|=|M-X|=|M|$.
\end{cor}
The non-trivial direction again follows by our results, since the cardinality principles are true in all standard infinite models.  

These are not, however, the best results
available, since a weaker condition on the right-hand side may also be sufficient for inconsistency (respectively, unsatisfiability).  We can already obtain such improvements understanding ``large enough'' to mean not ``universe-sized'', but ``exponentially large'' in the sense of Fine (see Appendix~\ref{sec:reFine}).  That is, replace ``for some $X$ with $|X|=|M-X|=|M|$'' on the right with ``for some $X$ such that $TOP(X)$'', and the corollaries still hold.  

In any case, as an explanation of the consistency and satisfiability profiles of abstraction principles it is worth considering.

Outside of the discussion on neo-logicism, we note that the classification theorems proved herein, being completely general, may have applications of interest in other research areas related to second-order logic.  Having answered some questions, it is thus apt to raise one in conclusion.    We showed in Theorem~\ref{thm:main-0formal} that for concepts finite in $\mathcal M$, the conditional arrows ``$\Rightarrow$'' reverse for both separative and complementative bicardinal slices.  Is this true for more than just finite concepts?  In other words, are there $L_0$-definable equivalence relations in which the arrow of implication does not reverse?

\section{Appendix:  The proof of Proposition~\ref{p:nonA}}\label{appendix:props}

\sloppy 
\begin{proof} 
Suppose $|X|\ntrianglelefteq|Y|$, so by Cardinal Comparability, for all $n\in \mathbb N$, $\mathcal M\models n\times |Y|<|X|$.  Thus $1\times |X| \geq |Y|$, so $|Y|\trianglelefteq |X|$, establishing (\ref{e:nonAcomp}).

For $X \cap Z = Y\cap W = \emptyset$ and $|Z|\stackrel{g}\leq|W|$, suppose $|X|\trianglelefteq|Y|$.  Choose $n$, $U$, and $f$ such that $|U|=n$ (in the metatheory), and $|X| \stackrel{f}\leq |U\times Y|$.  Then choose  $u\in U$, and then let 
\begin{equation}
h(v) = \begin{cases}
		f(v) & v \in X \\
		(u,g(v)) & v \in Z \\ 
\end{cases}
\end{equation}
The function $h$ is an injection establishing (\ref{e:nonAaddit}).

Toward establishing (\ref{e:nonAsubt}), assume $|Z|\leq|W|$ and  $|X\sqcup Z|\trianglelefteq |Y\sqcup W|$, but $\neg(|Y\sqcup W|\trianglelefteq|X\sqcup Z|)$.  For a contradiction, assume $\neg(|X|\vartriangleleft|Y|)$.  By (\ref{e:nonAcomp}), $|Y|\trianglelefteq|X|$.  But then since $|W|\leq|Z|$, by (\ref{e:nonAaddit}) we have $|Y\sqcup W|\trianglelefteq|X\sqcup Z|$, contradicting our assumption.

\fussy

For (\ref{e:nonAtrans}), choose $n, m$ such that $|X|\stackrel{f}\leq n\times |Y|$ and $|Y|\stackrel{g}\leq m\times |Z|$.  Then define $h:X \rightarrow nm(n+1)\times Z$ by 
\begin{equation}
h(x) = (nmf_1(x) + g_1(f_1(x)), g_2(f_2(x))
\end{equation}
where $f_1$ and $g_1$ output the left value of $f$ and $g$ respectively, and $f_2$ and $g_2$ output the right value of $f$ and $g$, respectively.  Verifying that $h$ is injective is routine, $h\in \mathcal M$ by comprehension.

To establish (\ref{e:nonAsplit}), suppose 
$\mathcal M\models X\stackrel{f}\leq n \times (Y\sqcup Z)$.  By (\ref{e:nonAcomp}), either $Y\trianglelefteq Z$ or $Z\trianglelefteq Y$;  assume the former, letting $g$ witness the injection in $\mathcal M$.  Then 
\begin{equation}
h(x) = \begin{cases}
		f(x) & f_2(x)\in Z \\
		(n+f_1(x), g_2(f_2(x))) & f_2(x)\in Y 
\end{cases}
\end{equation}
injects $X$ into $2n \times Z$; $h\in \mathcal M$ by comprehension.  So $X\trianglelefteq Z$;  a similar argument shows that if $Z\trianglelefteq Y$ then $X\trianglelefteq Y$.  Note that (\ref{e:nonAsplcor}) follows immediately.

For (\ref{e:nonApsplit}), suppose $\mathcal M\models |Y|\stackrel{f}\leq n\times |X|$ and $\mathcal{M}\models |Z|\stackrel{g}\leq m\times |X|$.  Then 
\begin{equation}
h(u) = \begin{cases}
		f(u) & u\in Y \\
		(n+i, x) & u \in Z, g(u) = (i,x) \\
\end{cases}
\end{equation}
injects $|Y\sqcup Z|$ into $(n+m)\times |X|$ in $\mathcal M$.  So if $|X|\vartriangleleft |Y\sqcup Z|$, $|Y\sqcup Z|\ntriangleleft |X|$, and so $|Y|\ntrianglelefteq|X|$ or $|Z|\ntrianglelefteq|X|$.  By (\ref{e:nonAcomp}) this implies that either $|X|\vartriangleleft|Y|$ or $|X|\vartriangleleft |Z|$ holds.  Note that (\ref{e:nonApslcor}) follows immediately.

To establish (\ref{e:nonAlsplit}), assume 
$|Y\sqcup Z|\stackrel{f}\trianglelefteq|X|$, then  $f\upharpoonright Y$ and $f\upharpoonright Z$ are the required injections.  If $|X|\trianglelefteq |Y|$ then by (\ref{e:nonAaddit}) and (\ref{e:nonAtrans}) $|X|\trianglelefteq |Y \sqcup Z|$;  so (\ref{e:nonAplsplit}) follows from (\ref{e:nonAlsplit}).

For (\ref{e:nonAexp}), assume $|X|\vartriangleleft|Y|$, so by the definition of $\vartriangleleft$, we have that $|Y|\ntrianglelefteq|X|$.  Thus by (\ref{e:nonAlsplit}),  $|Y\sqcup Z|\ntrianglelefteq|X|$.  Then by (\ref{e:nonAcomp}), $|X|\vartriangleleft |Y\sqcup Z|$.

For (\ref{e:nonAsubl}), assume $|X\sqcup Y|\stackrel{f}\vartriangleleft |Z|$ and $|W|\stackrel{g}=|Y|$.  Then 
\begin{equation}
h(u) = \begin{cases}
f(u) & u\in X \\
f(g(u)) &u\in W 
\end{cases}
\end{equation}
witnesses $|X\sqcup W|\vartriangleleft |Z|$.  A similar argument establishes (\ref{e:nonAsubr}).  Then (\ref{e:nonAsublp}) follows from (\ref{e:nonAsubl}), (\ref{e:nonAsubr}) and (\ref{e:nonAcomp}), as does (\ref{e:nonAsubrp}).
\end{proof} 

\section{Appendix:  Fine's Characterization Theorem}\label{sec:reFine}

We have noted throughout that our Main Theorem is a version of Fine's Classification Theorem;  additionally Theorem~\ref{thm:BC} bears a striking similarity to Fine's Characterization Theorem~\cite[Theorem~6, p.~144]{Fine2002}.  In this last appendix we show how to present Fine's Characterization Theorem in the deductive setting of our paper.  

Following Fine~\cite[p.~143]{Fine2002}, say that a concept is \emph{exponentially large} in $\mathcal M$ just its subconcepts outnumber the objects of $M$.  For short-hand, now say that a concept $X$ is ``Top'' just if $X$ and its complement are both exponentially large.  The \emph{basal} equivalence relation $E_0(X,Y)$ is the equivalence relation such that for any $\mathcal M$ and $X,Y\in \mathcal M$, $E(X,Y)$ holds (in the metatheory)\footnote{Fine's results are cast in terms of partitions of $P(M)$, so at this point we won't talk about $\mathcal M$ satisfying $E_0(X,Y)$, since Fine's presentation doesn't ensure that $E_0$ is expressible.  We will show how to express something like $E_0$ in $L_0$ below, see Definition~\ref{def:basal}.} just if either both $X$ and $Y$ are Top and $\mathcal M \models X\boxminus Y$, or neither $X$ nor $Y$ is Top and $\mathcal M \models X=Y$.  
\begin{FFCT}
The basal relation $E_0$ is the finest equivalence relation satisfying Permutation Invariance such that for any infinite \emph{standard} model $\mathcal M$, $M\models A_{E_0}$.  
\end{FFCT}
The most noticable differences between Fine's Characterization Theorem and Theorem~\ref{thm:BC} are that the latter concerns only infinite standard models, while the former addresses all  models with well-behaved cardinalities.
Fine's proof of the Characterization theorem uses his Classification theorem.  As we have stressed, this classification theorem is a version of what we have proved as our Main Theorem.  But it is hard to see how, given that Fine's terminology doesn't neatly capture the array of possibilities.  But some reflection, with Fine's suggested aid of a Venn Diagram, shows that the Fine's Classification theorem can be restated as follows.  
\begin{FClass}[Restated]
For $X$ infinite and $\mathcal M$ standard: 

(\ref{i:fclassi}) If $|X|<|M|$ and $E(\boxminus)_X$ is does not refine a separation, then $E(\boxminus)_X$ is trivial. 

(\ref{i:fclassii}) If $|X|=|M|>|M-(X\cup Y)|$ and $E(\boxminus)_X$ refines neither a separation nor a complementation, then $E(\boxminus)_X$ is trivial. 

(\ref{i:fclassiii}) If $|X|=|M-X|$ and $E(\boxminus)_X$ refines neither a separation nor a complementation, then $E(\boxminus)_X$ is trivial.
\end{FClass}
And this yields as a corollary our Main Theorem, restricted to infinite standard models:  that $E(\boxminus)_X$ must satisfy at least one of being trivial, refining a separation, or refining a complementation.  

Our Theorem~\ref{thm:main-formal} can thus be regarded as 
sharpening and expanding Fine's achievement in his Classification Theorem, except for one easily ironed wrinkle:  
\begin{rmk}
\label{rmk:wrinkle} In the statement of our theorem:  our main concern has been with identifying $L_0$-definability as a necessary condition on an equivalence relation's being logical, so our conditions on $E$ have been that it be $L_0$-definable.  However, inspection of the proofs of our results will show that $L_0$-definability is used only in invoking Permutation Invariance.   
Thus all of our results can be put as Fine's are:  about equivalence relations satisfying Permutation Invariance.  
\end{rmk} 

As we will show below, Fine's Characterization Theorem can also be sharpened and expanded, with the help of Theorem~\ref{thm:main-formal}.  We prove the Characterization Theorem in the deductive setting of Theorem~\ref{thm:main-formal}.

To do so, we first need to use $L_0$, rather than the metatheory, to describe concepts being exponentially large, ``Top'', and so to describe the basal equivalence relation $E_0$.  Given
 a relation $R(x,y)$, we let $R[x] = \{y \mid R(x,y)\}$.  If $R\in \mathcal M$ then so is $R[x]$ by comprehension.  Using this we can express, in $L_0$, the claim that  $R$ ``injects''
 from equivalence classes of subconcepts of a given concept $W$ to objects with the following formula (see also~\cite[p.~588]{WalshED2015}, \cite[p.~105]{Shapiro1991}): 
\begin{multline}
\label{e:R-inj-eqclass-objects}
 (\forall U\subseteq W)(\exists x)(R[x]\subseteq U \wedge E(R[x], U)) \wedge {} \\ 
				 (\forall x,y)(\neg E(R[x], R[y]) \rightarrow x\neq y)
\end{multline}
We will use the expression ``$\left| \frac{S_1[M]\upharpoonright W}{E}\right| \stackrel{R}\leq |M|$'' to abbreviate (\ref{e:R-inj-eqclass-objects}).\footnote{It is worth noting that since we are allowing non-standard models, for some $E$ there may be ``false negatives'' (where $W = M$ we omit the restriction ``${} \upharpoonright M$''):  
There are models $\mathcal M$ and equivalence relations $E$ such that 
\[ \mathcal M \models \neg \left(\exists R\right)\left( \left| \frac{S_1[M]}{E}\right| \stackrel{R}\leq |M|\right) \]
even though there is, in the metatheory, an injection from the $E$-partition of $S_1[M]$ to $M$.
A sketch of the proof is as follows:  

The witnessing equivalence relation will be $\approx$.  
Let $L$ be a language expanding $L_0$ by countably many constants $c_i$ for first-order objects, and the $T$ be the theory containing as axioms all comprehension axioms in the signature, as well as all sentences of the form: 
\begin{align*}
c_i& \neq c_j & \textrm{for each $i\neq j$} \\
 \neg(\exists R)&(\varphi (R) \wedge \left| \frac{S_1[M]}{\approx}\right| \stackrel{R}\leq |M|) & \textrm{for each formula $\varphi(R)$ of $L$} \\
\end{align*}
Every finite subset of $T$ is consistent (and in fact satisfiable), so $T$ is consistent, and has a model whose first-order domain $M$ is infinite \emph{in the metatheory}.    The resulting model witnesses the truth of the theorem. 
}
We now have the material required for expressing the basal equivalence relation:
\begin{defn}
\label{def:basal}  Given $E$ an $L_0$-definable equivalence relation, abbreviate as follows:
\begin{align}
EXPL(X)  := & \neg \left| \frac{S_1[M]\upharpoonright X}{=}\right| \leq |M| \\
TOP(X)  : = &  EXPL(X) \wedge EXPL(M-X) \\
E_0(X,Y)  := &  (TOP(X) \wedge TOP(Y) \wedge X \boxminus Y) \vee {} \\
& \nonumber (\neg (TOP(X) \vee TOP(Y)) \wedge X=Y)
\end{align}
\end{defn}
The abstraction principle $A_{E_0}$ is the correlate of Fine's basal abstraction principle.  
\begin{thm} [Generalization of Fine's Characterization Theorem]
\label{thm:sharpFine1}
The abstraction principle $A_{E_0}$ is consistent, and for any expressible equivalence relation $E$ satisfying Permutation Invariance on all models, if $A_E$ is consistent then $E_0$ is finer than $E$.

(Recall that $E_0$ is finer than $E$ means that $\models (\forall X, Y)(E_0(X,Y) \rightarrow E(X,Y)$.)
\end{thm}
Note that Theorem~\ref{thm:sharpFine1} implies Fine's Characterization Theorem.  
\begin{proof}
We give only the portion of the proof not available elsewhere.  To see that $A_{E_0}$ is consistent, see~\cite[p.~144]{Fine2002}. 
Towards establishing the second assertion, note that extensional equality of concepts is the finest Permutation Invariant equivalence relation (\cite[Theorem 2, p.~281]{Antonelli2010aa} also uses this fact).  So let $\mathcal M \models A_E$;  if $E$ is strictly finer than $E_0$, there are $W,S$ such that $\mathcal M \models TOP(W) \wedge TOP(S)$, $W\boxminus S$ and $\neg E(W,S)$.  Thus $E(\boxminus)_W$ is non-trivial, and by Theorem~\ref{thm:main-formal}, it either refines a separation or refines a complementation.  

If it refines a complementation then $|W|=|M-W|$.  If $|M|\leq 4$ then 
\[ \mathcal M\models (\forall X,Y)(E_0(X,Y)\leftrightarrow X \approx Y) \]
and so $A_{E_0}$ is {\tt HP}.  But {\tt HP} has no models of finite size, and since $E$ is finer than $E_0$, $A_E$ has no finite models either.  Thus $M$ is infinite.  But then by Theorem~\ref{thm:BC}, $\mathcal M \not \models A_E$.

If $E$ refines a separation, then by Lemma~\ref{lem:refinterp}, $A_E$ and $|W|\leq |M-W|$ translate ${\tt BLV}(\leeq W)$.  Set $R(x,y)$ to be $(\exists Y)(x = \varepsilon Y \wedge Yy)$; $R$ exists by comprehension and $R[\epsilon Y] = Y$.  It is easy to verify using ${\tt BLV}(\leeq W)$ implies 
\[  \left| \frac{S_1[M]\upharpoonright W}{=}\right| \stackrel{R}\leq |M| \] 
and this contradicts the assumption that $TOP(W)$.\footnote{Fine's Characterization results can in fact be generalized neatly to more types of invariance, as Fine does himself;  see~\cite[Corollary 7, p.~146]{Fine2002}, and \cite{Cook2016doi}. 
     }
  
\end{proof}

\section*{Acknowledgements}

The author wishes to thank Sean Walsh, Roy Cook, Jack Woods, and Eileen Nutting for helpful discussions, comments, and encouragement on this paper.  All mistakes, however, are my own.  The work on this paper was done much in anticipation of showing it to the author's late friend and former teacher, Aldo Antonelli.  That Aldo could not see the conclusion of his inspiration and support marks the completion of this project with great sadness.

\bibliographystyle{plain}
\bibliography{RC-bib}

\end{document}